\documentclass[11pt]{article}
\usepackage{amssymb,amsmath,accents}
\usepackage{amscd}
\usepackage{amsfonts,amsthm,mathrsfs}
\usepackage{setspace} 
\usepackage{graphics}
\usepackage{latexsym}
\usepackage{color}
\usepackage{showlabels}

\def\tcr{\textcolor{red}}

\newtheorem{theorem}{Theorem}
\newtheorem{lemma}[theorem]{Lemma}
\newtheorem{proposition}[theorem]{Proposition}

\theoremstyle{definition}
\newtheorem{definition}[theorem]{Definition}
\newtheorem{remark}[theorem]{Remark}

\title{\textbf{Normal trace for vector fields of bounded mean oscillation}}
\author{Yoshikazu Giga and Zhongyang Gu}

\begin{document}
\date{}

\maketitle
\begin{abstract}
We introduce various spaces of vector fields of bounded mean oscillation ($BMO$) defined in a domain so that normal trace on the boundary is bounded when its divergence is well controlled.
 The behavior of ``normal'' component and ``tangential'' component may be different for our $BMO$ vector fields.
 As a result zero extension of the normal component stays in $BMO$ although such property may not hold for tangential components.
\end{abstract}
%
\begin{center}
\end{center}

\section{Introduction} 

One of basic questions on vector fields defined on a domain $\Omega$ in $\mathbf{R}^n$ is whether the normal trace is well controlled without estimating all partial derivatives when the divergence is well controlled.
 Such a type of estimates is well known when a vector field is $L^p$ ($1<p<\infty$) or $L^\infty$.
 Here are examples.
 Let $\Omega$ be a bounded domain with smooth boundary $\Gamma$.
 Let $\mathbf{n}$ denotes its exterior unit normal vector field on $\Gamma$.
 For simplicity we assume that a vector field $v$ satisfies $\operatorname{div}v=0$.
 Then there is a constant $C$ independent of $v$ such that
\begin{align}
	\| v \cdot \mathbf{n} \|_{W^{-1/p, p}(\Gamma)} &\leq C \| v \|_{L^p(\Omega)} \label{N1} \\
	\| v \cdot \mathbf{n} \|_{L^\infty(\Gamma)} &\leq C \| v \|_{L^\infty(\Omega)}. \label{N2}
\end{align}
Here $W^{s, p}$ denotes the Sobolev space which is actually a Besov space $B^s_{p,p}$ for non-integer $s$.
 The first estimate is a key to establish the Helmholtz decomposition of an $L^p$ vector field; see e.g.\ \cite{FM}.
 The second estimate is important to study for example total variation flow; see e.g.\ \cite[Appendix C1]{ACM}.
 These estimates \eqref{N1}, \eqref{N2} hold for various domains including the case that $\Omega$ is a half space $\mathbf{R}^n_+$ i.e.,
\[
	\mathbf{R}^n_+ = \bigl\{ (x_1, \ldots, x_n) \bigm| x_n > 0 \bigr\}.
\]

Our goal in this paper is to extend \eqref{N2} by replacing $\|v\|_{L^\infty(\Omega)}$ by some $BMO$ type norm.
 However, it turns out that the normal trace of divergence free $BMO$ vector fields may not be bounded.
 Indeed, consider
\[
	v = (v^1, v^2), \quad
	v^1(x) = v^2(x) = \log|x_1-x_2|. \quad
	x = (x_1, x_2) \in \mathbf{R}^2.
\]
This vector field is in $BMO(\mathbf{R}^2)$ and it is divergence free in distribution sence.
 Indeed,
\[
	\int_{\mathbf{R}^2} v\cdot\nabla\varphi \, dx
	= \frac{1}{2} \int_{\mathbf{R}^2}  \log |\zeta| 
	\bigl( (\partial_\zeta - \partial_\zeta) \tilde{\varphi} + (\partial_\eta + \partial_\eta) \tilde{\varphi} \bigr) \, d\zeta d\eta = 0,
\]
\[
	\zeta = x_1 - x_2, \quad
	\eta = x_1 + x_2
\]
for all compactly supported smooth function $\varphi$, i.e., $\varphi \in C^\infty_c(\mathbf{R}^2)$.
 Here, $\tilde{\varphi} (\zeta, \eta) = \varphi \bigl( (\zeta+\eta)/2, (\eta-\zeta)/2 \bigr)$.
 However, if we consider $\Omega = \mathbf{R}^2_+$ and $\Gamma = \{ x_2 = 0 \}$, then $v \cdot \mathbf{n} = -v_2$ on $\Gamma$ is clearly unbounded.
 This example indicates that we need some control near the boundary.
 Such a control is introduced in \cite{BG}, \cite{BGS}, \cite{BGMST}, \cite{BGST}.
 More precisely, for $f \in L^1_{\mathrm{loc}}(\Omega)$, $\nu \in (0, \infty]$ they introduced a seminorm
\[
	[f]_{b^\nu} := \sup \left\{ r^{-n} \int_{\Omega\cap B_r(x)} \left| f(y) \right| \, dy \biggm| x \in \Gamma,\ 0<r<\nu \right\},
\]
where $B_r(x)$ denotes the closed ball of radius $r$ centered at $x$.
 For $\mu \in (0,\infty]$ they define
\[
	[f]_{BMO^\mu} := \sup \left\{ \frac{1}{\left| B_r (x) \right|} \int_{B_r(x)} \left| f - f_{B_r(x)} \right| \, dy \biggm| B_r(x) \subset \Omega,\ r<\mu \right\},
\]where $\displaystyle f_B = \frac{1}{|B|} \int_B f(y) \, dy$, the average over $B$; here $|B|$ denotes the Lebesgue measure of $B$.
 The $BMO$ type space $BMO^{\mu,\nu}_b$ introduced in these papers is the space of $f \in L^1_{\mathrm{loc}}(\Omega)$ having finite
\[
	\| f \|_{BMO^{\mu,\nu}_b} := [f]_{BMO^\mu} + [f]_{b^\nu}.
\]
This space is very convenient to study the Stokes semigroup in \cite{BG}, \cite{BGS}, \cite{BGMST}, \cite{BGST} as well as the heat semigroup \cite{BGST}.
One of our main results (Theorem \ref{TRBB}) yields
\begin{equation} \label{N3}
	\| v \cdot \mathbf{n} \|_{L^\infty(\Gamma)} \leq C \|v\|_{BMO^{\mu,\nu}_b}
\end{equation}
for any $\mu,\nu \in (0,\infty]$ for any uniformly $C^{1+\beta}$ domain.

However, for applications, especially to establish the Helmholtz decomposion $BMO^{\mu,\nu}_b$ norm for all components in too strong so we would like to estimate by a weaker norm.
 We only use $b^\nu$ seminorm for normal component of a vector field $v$.
 To decompose the vector field let $d_\Omega(x)$ be the distance of $x \in \Omega$ from the boundary $\Gamma$, i.e., 
\[
	d_\Omega(x) := \inf \bigl\{ |x-y| \bigm| y \in \Gamma \bigr\}. 
\]
If $\Omega$ is uniformly $C^2$, then $d_\Omega$ is $C^2$ in a $\delta$-tubular neighborhood $\Gamma_\delta$ of $\Gamma$ for some $\delta<R_*$, where $R_*$ is the reach of $\Gamma$ \cite[Chapter 14, Appendix]{GT}, \cite[\S4.4]{KP}; here
\[
	\Gamma_\delta := \bigl\{ x \in \Omega \bigm| d_\Omega(x) < \delta \bigr\}. 
\]
Instead of \eqref{N3}, our main results (Theorem \ref{NTH}, \ref{NTG}) together with Theorem \ref{V1} read:
\begin{equation} \label{MTB}
	\| v \cdot \mathbf{n} \|_{L^\infty(\Gamma)} \leq C \bigl( [v]_{BMO^\mu} + [\nabla d_\Omega \cdot v]_{b^\nu} \bigr)
\end{equation}
for $\nu \leq \delta$, $\mu \in (0,\infty]$ provided that $\Omega$ is a bounded $C^{2+\beta}$ domain with $\beta\in(0,1)$.
 The quantity $\nabla d_\Omega \cdot v$ is a kind of normal component.
 
Our main strategy is to use the formula
\[
	\int_\Gamma (v \cdot \mathbf{n}) \psi\, d\mathcal{H}^{n-1}
	= \int_\Omega (\operatorname{div}v) \rho\, dx
	- \int_\Omega v\cdot \nabla\varphi\, dx
\]
for any $\varphi \in C^\infty_c(\overline{\Omega})$ with $\varphi \, |_\Gamma = \psi$, where $d\mathcal{H}^{n-1}$ denotes the surface element.
 This formula is obtained by integration by parts.
 If $\operatorname{div} v=0$, then it reads
\begin{equation} \label{SINT}
	\int_\Gamma (v \cdot \mathbf{n}) \psi\, d\mathcal{H}^{n-1}
	= -\int_\Omega v\cdot \nabla\varphi\, dx.
\end{equation}
Our estimate \eqref{MTB} follows from localization, flattening the boundary and duality argument.
 To get the flavor we explain the case when $\Omega$ is the half space $\mathbf{R}^n_+$.
 For $\psi \in L^1(\Gamma)$ it is known that there is $\varphi \in F^1_{1,2}(\mathbf{R}^n)$ such that its trace to the boundary equals to $\psi$; see e.g. \cite[Section 4.4.3]{Tr92}.
 Here $F^1_{1,2}$ denotes the Triebel-Lizorkin space which means that $\nabla\varphi \in h^1$, a localized Hardy space.
 We may assume that $\varphi$ is even in $x_n$.
 We extend $v=(v', v_n)$ even in $x_n$ for tangential part $v'$ and odd in $x_n$ for the normal part $v_n=\nabla d_\Omega \cdot v$.
 Although extended $v'$ is still in $BMO^\infty(\mathbf{R}^n)$, the extended $v_n$ may not be in $BMO^\infty(\mathbf{R}^n)$ unless we assume $[v_n]_{b^\nu}<\infty$.
 Here we invoke $[\nabla d_\Omega \cdot v]_{b^\nu} < \infty$.
 By these extensions our \eqref{SINT} yields
\begin{equation} \label{KID}
	\int_\Gamma (v \cdot \mathbf{n}) \psi\, d\mathcal{H}^{n-1}
	= -\frac{1}{2} \int_{\mathbf{R}^n} v\cdot \nabla\varphi\, dx,
\end{equation}
where $v$ denotes the extended vector field.
 We apply $h^1$-$bmo$ duality \cite[Theorem 3.22]{Sa} for \eqref{KID} to get
\[
	\left| \int_\Gamma (v \cdot \mathbf{n}) \psi\, d\mathcal{H}^{n-1} \right|
	\leq C \|v\|_{bmo} \|\varphi\|_{F^1_{1,2}},
\]
where $bmo=BMO\cap L^1_\mathrm{ul}$ a localized $BMO$ space.
 Here $L^1_\mathrm{ul}$ denotes a uniformly local $L^1$ space; see Section \ref{PS} for details.
 Since $\|\varphi\|_{F^1_{1,2}} \leq C\|\psi\|_{L^1}$, this implies
\begin{align} \label{MTH}
	\| v \cdot \mathbf{n} \|_{L^\infty(\Gamma)}
	&\leq C \|v\|_{bmo(\mathbf{R}^n)} \\
	&\leq C \left( [v]_{BMO^\infty(\Omega)} 
	+ [v]_{L^1_\mathrm{ul}(\Omega)}
	+ [ \nabla d_\Omega \cdot v ]_{b^\infty} \right). 
\end{align}
Here and hereafter $C$ denotes a constant depends only on space dimension and its numerical value may be different line by line.

In the case of a curved domain we need localization and flattening procedure by using normal (principal) coordinates.
 The localized space $bmo^\mu_\delta=BMO^\mu \cap L^1_\mathrm{ul}(\Gamma_\delta)$ is convenient for this purpose.
 Again we have to handle normal component $\nabla d_\Omega \cdot v$ separately.
 If the domain has a compact boundary, we are able to remove $L^1_\mathrm{ul}$ term in \eqref{MTH} and we deduce the estimate \eqref{MTB}.
 Note that in this trace estimate only the behavior of $v$ near $\Gamma$ is important so one may use finite exponents in $BMO^\mu$ and $b^\nu$.

As a byproduct we notice the extension problem of $BMO$ functions.
 In general, zero extension of $v \in BMO^{\mu}(\Omega)$ may not belong to $BMO^\mu(\mathbf{R}^n)$ but if $v$ is in $BMO^{\mu,\nu}_b$, as noticed in \cite{BGST}, its zero extension belongs to $BMO^\mu(\mathbf{R}^n)$ for $\nu \geq 2\mu$.
 We also note that it is possible to extend general $bmo^\mu_\delta(\Omega)$ to $BMO^\mu$ whose support is near $\overline{\Omega}$.
 We develop such a theory to explain the role of $b^\nu$.

This paper is organized as follows. In Section \ref{PS} we introduce several localized $BMO$ spaces and compared these spaces.
 Some of them are discussed in \cite{BGST}.
 We introduce a new space $vbmo^{\mu,\nu}_\delta$ which requires that the $b^\nu$ seminorm of the normal component is bounded in $(bmo^\mu_\delta)^n$.
 A key observation is that if the boundary of the domain is compact, i.e., either a bounded or an exterior domain, the requirement in $L^1_\mathrm{ul}(\Gamma_\delta)$ is redudent in the definition of $vbmo^{\mu,\nu}_\delta$.
 In Section \ref{JE} we discuss extension problem as well as localization problem.
 In Section \ref{TR} we shall prove our main results.
 In Appendix we discuss coordinate change of vector fields by normal coordinates for the reader's convenience.

\section{Spaces} \label{PS} 

In this section we fix notation of important function spaces.
 Let $L^1_{\mathrm{ul}}(\mathbf{R}^n)$ be a uniformly $L^1$ space, i.e., for a fixed $r_0>0$
\[
	L^1_{\mathrm{ul}}(\mathbf{R}^n) := 
	\left\{ f \in L^1_{\mathrm{loc}} (\mathbf{R}^n) \biggm| 
	\| f \|_{L^1_{\mathrm{ul}}} := \sup_{x \in \mathbf{R}^n} \int_{B_{r_0}(x)} \bigl| f(y) \bigr| \, dy < \infty \right\}.
\]
The space is independent of the choice of $r_0$.
 For a domain $\Omega$ the space $L^1_{\mathrm{ul}}$ is the space of all $L^1_{\mathrm{loc}}$ functions $f$ in $\Omega$ whose zero extension belongs to $L^1_{\mathrm{ul}}(\mathbf{R}^n)$.
 In other words
\[
	L^1_{\mathrm{ul}}(\Omega) := 
	\left\{ f \in L^1_{\mathrm{loc}} (\Omega) \biggm| 
	\| f \|_{L^1_{\mathrm{ul}}(\Omega)} := \sup_{x \in \mathbf{R}^n} \int_{B_{r_0}(x)\cap\Omega} \bigl| f(y) \bigr| \, dy < \infty \right\}.
\]
As in \cite{BG} we set
\[
	BMO^\mu (\Omega) := 
	\left\{ f \in L^1_{\mathrm{loc}} (\Omega) \biggm| 
	[f]_{BMO^\mu} < \infty \right\}.
\]
For $\delta \in (0,\infty]$ we set
\[
	bmo^\mu_\delta (\Omega) := BMO^\mu (\Omega) \cap L^1_{\mathrm{ul}} (\Gamma_\delta) = 
	\left\{ f \in BMO^\mu (\Omega) \bigm| 
	\text{restriction of } f \text{ on } \Gamma_\delta \text{ is in }  L^1_{\mathrm{ul}}(\Gamma_\delta) \right\}.
\]
This is a Banach space equipped with the norm
\[
	\| f \|_{bmo^\mu_\delta} := [f]_{BMO^\mu(\Omega)} + [f]_{\Gamma_\delta}, \quad
	 [f]_{\Gamma_\delta} := \| f \|_{L^1_{\mathrm{ul}}(\Gamma_\delta)},
\]
where the restriction of $f$ on $\Gamma_\delta$ is still denoted by $f$.
 If there is no boundary we set
\[
	bmo(\mathbf{R}^n) := BMO^\infty (\mathbf{R}^n) \cap L^1_{\mathrm{ul}}(\mathbf{R}^n)
\] 
which is a local $BMO$ space and it agrees with the Triebel-Lizorkin space $F^0_{\infty, 2}$; see e.g,\ \cite[Section 1.7.1]{Tr92}, \cite[Theorem 3.26]{Sa}.

For vector-valued function spaces we still write $BMO^\mu$ instead of $(BMO^\mu)^n$.
 For example for vector field $v$ by $v \in bmo^\mu_\delta(\Omega)$ we mean that
\[
	v = (v_1, \ldots, v_n), \quad 
	v_i \in bmo^\mu_\delta(\Omega), \quad
	1 \leq i \leq n.
\]
We next introduce the space of vector fields whose normal component has finite $b^\nu$ of the form
\[
	vbmo^{\mu, \nu}_\delta(\Omega) := \left\{ v \in bmo^\mu_\delta (\Omega) \bigm| 
	[\nabla d_\Omega \cdot v]_{b^\nu} < \infty \right\}
\]
for $\nu \in (0,\infty]$.
 This space is a Banach space equipped with the norm
\[
	\| v \|_{vbmo^{\mu, \nu}_\delta} := \| v \|_{bmo^\mu_\delta} + [\nabla d_\Omega \cdot v]_{b^\nu}. 
\]
Similarly, we introduce another space
\[
	vBMO^{\mu, \nu}(\Omega) := \left\{ v \in BMO^\mu (\Omega) \bigm| 
	[\nabla d_\Omega \cdot v]_{b^\nu} < \infty \right\}
\]
equipped with a seminorm
\[
	[v]_{vBMO^{\mu, \nu}} := [v]_{BMO^\mu} + [\nabla d_\Omega \cdot v]_{b^\nu}. 
\]
Of course, this is strictly larger than the Banach space
\[
	BMO^{\mu, \nu}_b(\Omega) := \left\{ v \in BMO^\mu (\Omega) \bigm| 
	[v]_{b^\nu} < \infty \right\}
\]
equipped with the norm
\[
	\| v \|_{BMO^{\mu, \nu}_b} := [v]_{BMO^\mu} + [v]_{b^\nu}
\]
introduced essentially in \cite{BG}.
 Indeed, in the case when $\Omega$ is the half space $\mathbf{R}^n_+$
\begin{equation} \label{DC}
	vBMO^{\mu, \nu}(\mathbf{R}^n_+) = \left( BMO^\mu(\mathbf{R}^n_+) \right)^{n-1} \times BMO^{\mu, \nu}_b(\mathbf{R}^n_+),
\end{equation}
where in the right-hand side the each space denotes the space of scalar functions not of vector fields.
 This shows that $vBMO^{\mu, \nu}(\mathbf{R}^n_+)$ is strictly larger than $BMO^{\mu, \nu}_b(\mathbf{R}^n_+)$ for $n \geq 2$.

Although there are many exponents, the spaces may be the same for different exponents.
 By definition for $0 <\mu_1 \leq \mu_2 \leq \infty$, $0 < \nu_1 \leq \nu_2 \leq \infty$, $0 < \delta_1 \leq \delta_2 \leq \infty$, 
\[
	[f]_{BMO^{\mu_1}} \leq [f]_{BMO^{\mu_2}}, \quad 
	[f]_{b^{\nu_1}} \leq [f]_{b^{\nu_2}}, \quad 
	[f]_{\Gamma_{\delta_1}} \leq [f]_{\Gamma_{\delta_2}}.
\]
\begin{proposition} \label{E1}
Let $\Omega$ be an arbitrary domain in $\mathbf{R}^n$.
\begin{enumerate}
\item[(i)] Let $0<\mu_1<\mu_2<\infty$.
 Then seminorms $[\cdot]_{BMO^{\mu_1}}$ and $[\cdot]_{BMO^{\mu_2}}$ are equivalent. If $\Omega$ is bounded, one may take $\mu_2 = \infty$.
\item[(i\hspace{-0.1em}i)] Let $0<\delta_1<\delta_2<\infty$ and $\mu \in (0,\infty]$.
 Then there exists a constant $C>0$ depending only on $n$, $\mu$, $\delta_1$, $\delta_2$ and $\Omega$ such that
\[
	[f]_{\Gamma_{\delta_2}} \leq C \left( [f]_{BMO^\mu} + [f]_{\Gamma_{\delta_1}} \right).
\]
In particular, the norms $\| \cdot \|_{bmo_{\delta_1}^{\mu}}$ and $\| \cdot \|_{bmo_{\delta_2}^{\mu}}$ are equivalent. If $\Omega$ is bounded, one may take $\delta_2 = \infty$.
\end{enumerate}
\end{proposition}
\begin{proof}
(i) This is \cite[Theorem 4]{BGST} which follows from \cite[Theorem 3]{BGST}.

(i\hspace{-0.1em}i) Since the space $L_{\mathrm{ul}}^1 (\Gamma_\delta)$ is independent of the radius $r_0$ in its definition, without loss of generality, we may assume that $r_0 > \delta_{1}$. Let us firstly consider the case where the dimension $n > 1$. Let $k$ be the smallest integer such that $2^{-k} < \frac{\delta_{1}}{\sqrt{n}}$ and $x \in \mathbf{R}^{n}$. Notice that
\begin{eqnarray*}
\int_{B_{r_0}(x) \cap \Gamma_{\delta_2}} |f| \, dy = \int_{B_{r_0}(x) \cap \Gamma_{\delta_1}} |f| \, dy + \int_{B_{r_0}(x) \cap ( \Gamma_{\delta_2} \setminus \Gamma_{\delta_1} )} |f| \, dy,
\end{eqnarray*}
we can estimate $\| f \|_{L^1(B_{r_0}(x) \cap \Gamma_{\delta_1})}$ directly by $[f]_{\Gamma_{\delta_1}}$. Assume that $\Gamma_{\delta_2} \setminus \Gamma_{\delta_1} \neq \emptyset$. Let $D_{k}(x)$ be the set of dyadic cubes of side length $2^{-k}$ that intersect with $B_{r_0}(x) \cap ( \Gamma_{\delta_2} \setminus \Gamma_{\delta_1} )$. For a dyadic cube $Q_j \in D_{k}(x)$, we define $B_j$ to be the ball which has radius $\frac{\sqrt{n}}{2} \cdot 2^{-k}$ and shares the same center with $Q_j$. Let $C_{k}(x) := \{ B_j \, | \, Q_j \in D_{k}(x) \}$ and $\Sigma := \{ x \in \Omega \, | \, d_{\Omega}(x) = \delta_1 \}$. 

For $Q_j \in D_{k}(x)$ that intersects $\Sigma$, we seek to estimate $\| f \|_{L^1(B_j)}$. Let $c_j$ be a point on $\Sigma \cap Q_j$, we have that $B_{\delta_1} (c_j) \subset \Omega$. Indeed as otherwise, there exists $z \in B_{\delta_1}(c_j) \cap \Omega^{c}$. Then the line segment joining $c_j$ and $z$ must intersect $\Gamma$ at some point, say $z^{\ast}$. Then $| z^{\ast} - c_j | \leq | z - c_j | < \delta_1$. This contradicts the fact that $d_{\Omega}(c_j) = \delta_1$. For $y \in B_j$, $|y - c_j | < \sqrt{n} \cdot \ell(Q_j) = \sqrt{n} \cdot 2^{-k} < \delta_1$. So $B_j \subset B_{\delta_1}(c_j)$. Let $d_j \in \Gamma$ be a point such that $|c_j - d_j | = \delta_1$, then on the line segment joining $c_j$ and $d_j$, we can find a point $o_j$ such that $| o_j - d_j | = \frac{\sqrt{n}}{2} \cdot 2^{-k}$. For $y \in B_{\frac{\sqrt{n}}{2} \cdot 2^{-k}}(o_j)$, we have that $| d_{\Omega}(y) - d_{\Omega}(o_j) | \leq | y - o_j |$. Hence $d_{\Omega}(y) \leq d_{\Omega}(o_j) + | y - o_j | < \sqrt{n} \cdot 2^{-k} < \delta_1$. This means that $B_{\frac{\sqrt{n}}{2} \cdot 2^{-k}}(o_j) \subset \Gamma_{\delta_1}$. Moreover,
\begin{eqnarray*}
|c_j - y| \leq |c_j - o_j| + |o_j - y| \leq \delta_1 - \frac{\sqrt{n}}{2} \cdot 2^{-k} + \frac{\sqrt{n}}{2} \cdot 2^{-k} = \delta_1.
\end{eqnarray*}
Thus $B_{\frac{\sqrt{n}}{2} \cdot 2^{-k}}(o_j) \subset B_{\delta_1}(c_j)$. Denote $B_{\frac{\sqrt{n}}{2} \cdot 2^{-k}}(o_j)$ by $B_j^{\ast}$. We have that
\begin{eqnarray*}
\int_{B_j} |f| \, dy \leq \int_{B_{\delta_1}(c_j)} |f - f_{B_{\delta_1}(c_j)}| \, dy + \int_{B_{\delta_1}(c_j)} | f_{B_{\delta_1}(c_j)} - f_{B_j^{\ast}}| \, dy + \int_{B_{\delta_1}(c_j)} |f_{B_j^{\ast}}| \, dy.
\end{eqnarray*}
Notice that
\begin{eqnarray*}
\int_{B_{\delta_1}(c_j)} |f - f_{B_{\delta_1}(c_j)}| \, dy &\leq& C_n \cdot \delta_1^n \cdot [f]_{BMO^{\mu}}, \\
\int_{B_{\delta_1}(c_j)} | f_{B_{\delta_1}(c_j)} - f_{B_j^{\ast}} | \, dy &\leq& \frac{|B_{\delta_1}(c_j)|^2}{|B_j^{\ast}|} \cdot [f]_{BMO^{\mu}}, \\
\int_{B_{\delta_1}(c_j)} |f_{B_j^{\ast}}| \, dy &\leq& \frac{|B_{\delta_1}(c_j)|}{|B_j^{\ast}|} \cdot [f]_{\delta_1}.
\end{eqnarray*}
Since $|B_{\delta_1}(c_j)| = C_n \cdot \delta_1^{n}$ and $\frac{|B_{\delta_1}(c_j)|}{|B_j^{\ast}|} = \frac{C_n \cdot \delta_1^n}{(\frac{\sqrt{n}}{2} \cdot 2^{-k})^n} \leq \frac{C_n \cdot \delta_1^n}{(\frac{\delta_1}{4})^n} = C_n$, $\| f \|_{L^{1}(B_j)}$ is therefore controlled by $C_{\delta_1,n} \cdot \left( [f]_{BMO^{\mu}} + [f]_{\Gamma_{\delta_1}} \right)$.

Next we consider $Q_j^{'} \in D_k(x)$ that does not intersect $\Sigma$. Suppose that $Q_j \in D_k(x)$ has a touching edge with $Q_j^{'}$. There exists a ball $B_i^j$ of radius $\frac{\sqrt{n} - 1}{2} \cdot 2^{-k}$ which is contained in $B_j \cap B_j^{'}$ where $B_j, B_j^{'}$ are the smallest balls that contain $Q_j, Q_j^{'}$ respectively. Similar to above, as $B_i^j \subset B_j$,
\begin{eqnarray*}
\int_{B_j^{'}} |f| \, dy &\leq& \int_{B_j^{'}} |f - f_{B_j^{'}}| \, dy + \int_{B_j^{'}} |f_{B_j^{'}} - f_{B_i^j}| \, dy + \int_{B_j^{'}} |f_{B_i^j}| \, dy \\
&\leq& |B_j^{'}| \cdot [f]_{BMO^{\mu}} + \frac{|B_j^{'}|^2}{|B_i^j|} \cdot [f]_{BMO^{\mu}} + \frac{|B_j^{'}|}{|B_i^j|} \cdot \int_{B_j} |f| \, dy.
\end{eqnarray*}
Therefore if $\| f \|_{L^{1}(B_j)}$ is controlled by $C_{\delta_1,n} \cdot \left( [f]_{BMO^{\mu}} + [f]_{\Gamma_{\delta_1}} \right)$, $\| f \|_{L^{1}(B_j^{'})}$ is also controlled by $C_{\delta_1,n} \cdot \left( [f]_{BMO^{\mu}} + [f]_{\Gamma_{\delta_1}} \right)$.

Since $B_{r_0}(x) \cap ( \Gamma_{\delta_2} \backslash \Gamma_{\delta_1} )$ is connected, we can estimate $\| f \|_{L^{1}(B_j)}$ for every $Q_j \in D_k(x)$ where $B_j$ is the smallest ball that contains $Q_j$. For each $Q_j \in D_{k}(x)$, there exists $y \in Q_j \cap B_{r_0}(x)$, so for any $z \in B_j$, $|z -x| \leq |z - y| + |y - x| < \sqrt{n} \cdot 2^{-k} + r_0 < r_0 + \delta_1$. Thus $\underset{Q_j \in D_{k}(x)}{\bigcup} B_j \subset B_{r_0+\delta_1} (x)$. Let $N(D_{k}(x))$ be the number of cubes in $D_{k}(x)$, we have that
\begin{eqnarray*}
N(D_{k}(x)) \leq \frac{|B_{r_0+\delta_1}(x)|}{2^{-kn}} \leq C_n \cdot \left(\frac{r_0+\delta_1}{\delta_1}\right)^n.
\end{eqnarray*}
Therefore,
\begin{eqnarray*}
\int_{B_{r_0}(x) \cap (\Gamma_{\delta_2} \setminus \Gamma_{\delta_1})} |f| \, dy &\leq& \underset{B_j \in C_{r_0}(x)}{\sum} \int_{B_j} |f| \, dy \\
&\leq& N(D_{k}(x)) \cdot C_{\delta_1,n} \cdot \left( [f]_{BMO^{\mu}} + [f]_{\Gamma_{\delta_1}} \right) \\
&\leq& C_{n,\delta_1,r_0} \cdot \left( [f]_{BMO^{\mu}} + [f]_{\Gamma_{\delta_1}} \right).
\end{eqnarray*}

For the case where the dimension $n = 1$, we let $k$ to be the smallest integer such that $2^{-k} < \frac{\delta_1}{2}$ and $D_k$ to be the set of dyadic cubes of side length $2^{-k}$ that intersects $\Gamma_{\delta_2} \backslash \Gamma_{\delta_1}$. Notice that the region $\Gamma_{\delta_2} \backslash \Gamma_{\delta_1}$ is indeed a union of intervals. Without loss of generality, we can assume $\Omega$ to be $(0,\infty)$ and take $\mu = \infty$ by part (i) of this proposition. Thus in this case $\Gamma_{\delta_2} \backslash \Gamma_{\delta_1} = (\delta_1,\delta_2)$. For $Q_0 \in D_k$ such that $\delta_1 \in Q_0$,
\begin{align*}
\int_{Q_0} |f| \, dy &\leq \int_{2Q_0} |f| \, dy \leq \int_{2Q_0} | f - f_{2Q_0} | \, dy + \int_{2Q_0} | f_{2Q_0} - f_{Q_0^{\ast}} | \, dy + \int_{2Q_0} | f_{Q_0^{\ast}} | \, dy \\
&\leq C \cdot \left( [f]_{BMO^\infty} + [f]_{\Gamma_{\delta_1}} \right),
\end{align*}
where $Q_0^{\ast} = 2Q_0 \backslash \left( Q_0 \cup [\delta_1,\infty) \right)$ and $\ell(Q_0^{\ast}) = \frac{1}{2} \ell(Q_0) = 2^{-(k+1)}$. 

We then put an ordering on the elements of $D_k$ in the following way. For $j \in \mathbf{N}$, suppose that we have ordered intervals $Q_0, Q_1, ... , Q_{j-1}$, we pick $Q_j \in D_k \backslash \{Q_0, Q_1, ... , Q_{j-1} \}$ such that $Q_j$ has a touching edge with $Q_{j-1}$. For $Q_j \in D_k$, similarly we have that
\begin{align*}
\int_{Q_j} |f| \, dy &\leq \int_{2Q_j} |f| \, dy \leq \int_{2Q_j} | f - f_{2Q_j} | \, dy + \int_{2Q_j} | f_{2Q_j} - f_{Q_j^{\ast}} | \, dy + \int_{2Q_j} | f_{Q_j^{\ast}} | \, dy \\
&\leq C \cdot \left( [f]_{BMO^\infty} + [f]_{\Gamma_{\delta_1}} \right),
\end{align*}
where $Q_j^{\ast} = 2Q_{j-1} \cap 2Q_j$ and $\ell(Q_j^{\ast}) = \ell(Q_j) = 2^{-k}$.

Let $N(D_k)$ be the number of elements of $D_k$, we have that
\[
N(D_k) \leq \frac{\delta_2 - \delta_1}{2^{-k}} +2 \leq \frac{4(\delta_2 - \delta_1)}{\delta_1} +2
\]
and therefore
\[
\int_{\Gamma_{\delta_2} \backslash \Gamma_{\delta_1}} |f| \, dy \leq C_{\delta_2,\delta_1} \cdot \left( [f]_{BMO^{\mu}} + [f]_{\Gamma_{\delta_1}} \right).
\]
The proof is now complete.
\end{proof}

By this observation when we discuss the space $bmo^\mu_\delta$ there are only four types of spaces
\[
	bmo^\mu_\delta, \quad bmo^\infty_\delta, \quad bmo^\mu_\infty, \quad bmo^\infty_\infty
\]
for finite $\mu,\delta>0$.
 If $\Omega$ is bounded, it is clear that these four spaces agree with each other.
 However, if $\Omega$ is unbounded, there four spaces may be different because they requires different growth at infinity.
 Indeed, if $\Omega=(0,\infty)$
\[
	bmo^\infty_\infty \subsetneq bmo^\infty_\delta
\]
since $\log(x+1) \in bmo^\infty_\delta$ while it does not belong to $bmo^\infty_\infty$.
 Moreover, since $x \in bmo^\mu_\delta$ but it does not belong to neither $bmo^\mu_\infty$ nor $bmo^\infty_\delta$, we see that
\[
	bmo^\infty_\delta \subsetneq bmo^\mu_\delta, \quad
	bmo^\mu_\infty \subsetneq bmo^\mu_\delta.
\]
It is possible to prove that $bmo^\infty_\infty = bmo^\mu_\infty$. Indeed as $bmo_{\infty}^{\infty}(\Omega) \subset bmo_{\infty}^{\mu}(\Omega)$ is simply by the definition of the $BMO$ seminorm. It is sufficient to show the contrary, i.e. $[f]_{BMO^{\infty}} \leq C \cdot ( [f]_{BMO^{\mu}} + [f]_{\Gamma_{\infty}} )$. Without loss of generality, in defining the seminorm $[ \, \cdot \, ]_{L_{\mathrm{ul}}^1(\Gamma_{\infty})}$, we set the radius of the ball to be $\frac{\sqrt{n}}{2}$. For $B_{r}(x) \subset \Omega$ with $r< \mu$,
\begin{eqnarray*}
\frac{1}{|B_{r}(x)|} \int_{B_{r}(x)} | f - f_{B_{r}(x)} | \, dy \leq [f]_{BMO^{\mu}}.
\end{eqnarray*}
For $B_{r}(x) \subset \Omega$ with $r \geq \mu$, if $r \leq \frac{\sqrt{n}}{2}$, then
\begin{eqnarray*}
\frac{1}{|B_{r}(x)|} \int_{B_{r}(x)} | f - f_{B_{r}(x)} | \, dy \leq \frac{2}{|B_{r}(x)|} \int_{B_{\frac{\sqrt{n}}{2}}(x) \cap \Omega} |f| \, dy \leq C_{\mu,n} \cdot [f]_{\Gamma_{\infty}}.
\end{eqnarray*}
If $r > \frac{\sqrt{n}}{2}$, $B_{r}(x)$ is contained in the cube $Q_r$ with center $x$ and side length $2([r]+1)$, here $[r]$ is the largest integer less than or equal to $r$. By dividing each side length of $Q_r$ equally into $2([r]+1)$ parts, we can divide the cube $Q_r$ into $(2[r]+2)^n$ subcubes of side length $1$. Let $S_{Q_r}$ be the set of these $(2[r]+2)^n$ subcubes of $Q_r$. For $Q_r^i \in S_{Q_r}$, let $B_r^i$ be the smallest ball that contains $Q_r^i$. Let $C_{Q_r} := \{ B_r^i \; | \; Q_r^i \in S_{Q_r} \}$. We have that
\begin{eqnarray*}
\int_{B_{r}(x)} |f| \, dy \leq \sum_{i=1}^{(2[r]+2)^n} \int_{B_r^i \cap \Omega} |f| \, dy \leq (2[r]+2)^n \cdot [f]_{\Gamma_{\infty}}.
\end{eqnarray*}
Since $r \geq \mu$, 
\begin{eqnarray*}
\frac{1}{|B_r(x)|} \int_{B_r(x)} | f - f_{B_r(x)} | \, dy \leq \frac{2}{|B_r(x)|} \int_{B_r(x)} |f| \, dy \leq C_{\mu,n} \cdot [f]_{\Gamma_{\infty}}.
\end{eqnarray*}
Therefore $bmo_{\infty}^{\infty} = bmo_{\infty}^{\mu}$ and thus $bmo^\mu_\infty \subsetneq bmo^\infty_\mu$.

We summarize these equivalences.
\begin{theorem} \label{B1}
Let $\Omega$ be an arbitrary domain in $\mathbf{R}^n$. Then
\[
bmo_{\infty}^{\infty}(\Omega) = bmo_{\infty}^{\mu}(\Omega) \subset bmo_{\delta}^{\infty}(\Omega) \subset bmo_{\delta}^{\mu}(\Omega)
\]
for finite $\delta, \mu > 0$. The inclusions can be strict when $\Omega$ is unbounded. If $\Omega$ is bounded, all four spaces are the same.
\end{theorem}

As a simple application of Proposition \ref{E1}, we conclude that the space $BMO^{\mu, \nu}_b$ is included in $bmo^\mu_\nu$ since $[f]_\nu \leq c[f]_{b^\nu}$ ($\nu<\infty$) with $c>0$ depending on $\nu$.
\begin{theorem} \label{B2}
Let $\Omega$ be an arbitrary domain in $\mathbf{R}^n$.
 For $\mu \in (0,\infty]$ the inclusion
\[
	BMO^{\mu, \nu}_b(\Omega) \subset bmo^\mu_\nu(\Omega)
\]
holds for $\nu \in (0,\infty)$.
\end{theorem}

Since $b^{\nu}$-seminorm controls boundary growth stronger than $L^1$ sense, this inclusion is in general strict even when $\Omega$ is bounded. Here is a simple example when $\Omega=(0,1)$. The $b^\nu$-seminorm of $f(x)=\log x$ is infinite but $\| f \|_{L^1(\Omega)}$ is finite.

We next discuss the space $vbmo^{\mu, \nu}_\delta$.
\begin{remark} \label{MY}
As proved in \cite[Theorem 9]{BGST}, if $\Omega$ is a bounded Lipschitz domain, the space $BMO_b^{\mu,\nu}$ $(\mu,\nu \in (0,\infty])$ agrees with the Miyachi $BMO$ space \cite{Miy} defined by
\begin{align*}
BMO^M(\Omega) &= \left\{ f \in L_{\mathrm{loc}}^1(\Omega) \; \middle| \; \| f \|_{BMO^M} < \infty \right\}, \\
\| f \|_{BMO^M} &:= [f]_{BMO^M} + [f]_{b^M}, \\
[f]_{BMO^M} &:= \mathrm{sup} \left\{ \frac{1}{|B_r(x)|} \int_{B_r(x)} | f - f_{B_r(x)} | \, dy \; \, \middle| \; B_{2r}(x) \subset \Omega \right\}, \\
[f]_{b^M} &:= \mathrm{sup} \left\{ \frac{1}{|B_r(x)|} \int_{B_r(x)} |f| \, dy \; \, \middle| \; B_{2r}(x) \subset \Omega \; \text{ and } \, B_{5r}(x) \cap \Omega^c \neq \emptyset \right\}.
\end{align*}
\end{remark}
\begin{proposition} \label{E2}
Let $\Omega$ be an arbitrary domain in $\mathbf{R}^n$.
 Let $0<\nu_1\leq \nu_2 \leq \delta \leq \infty$.
 Then there exists a constant $c>0$ depending only on $n$, $\nu_1$, $\nu_2$, $\delta$ such that
\[
	[\nabla d_\Omega \cdot v]_{b^{\nu_2}} \leq 
	[\nabla d_\Omega \cdot v]_{b^{\nu_1}} + c[v]_{\Gamma_\delta} 
\]
for all $v \in L^1_{\mathrm{ul}}(\Gamma_\delta)$.
\end{proposition}
\begin{proof}
We may assume that $\nu_1<\infty$.
 Let $Q_r(x)$ denote a cube centered at $x$ with side length $2r$,
 Since $|\nabla d|=1$ and $B_r(x)\subset Q_r(x)$, we see that
\begin{align*}
	[\nabla d_\Omega \cdot v]_{b^{\nu_2}} - [\nabla d_\Omega \cdot v]_{b^{\nu_1}}
	& \leq \sup \left\{ \frac{1}{r^n} \int_{B_r(x)\cap\Omega} |\nabla d_\Omega \cdot v| \, dy \biggm| x \in \partial\Omega,\ \nu_1\leq r < \nu_2 \right\} \\
	& \leq \sup \left\{ \frac{1}{r^n} \int_{Q_r(x)} |\tilde{v}| \, dy \biggm| x \in \partial\Omega,\ \nu_1\leq r \leq \nu_2 \right\}
\end{align*}
where $\tilde{v}$ denotes the zero extension of $v$ to $\mathbf{R}^n$.
Since $\nu_2\leq\delta$ so that $Q_r(x) \cap \Omega \subset \Gamma_\delta$, we see that
\[
	\sup_{x \in \partial\Omega} \int_{Q_r(x)} |\tilde{v}| \, dy
	\leq \| v \|_{L^1_{\mathrm{ul}}(\Gamma_\delta)}
	\quad\text{for}\quad \nu_1 \leq r \leq \nu_2
\]
provided that $\nu_2$ is finite by taking an equivalent norm of $L^1_{\mathrm{ul}}$; in fact we take $r_0=\sqrt{n}\;\nu_2$ 
 This implies that
\[
	[\nabla d_\Omega \cdot v]_{b^{\nu_2}} - [\nabla d_\Omega \cdot v]_{b^{\nu_1}}
	\leq \frac{1}{\nu^n_1}[v]_{\Gamma_\delta}.
\]
If $\nu_2=\delta=\infty$, we may assume $r=2^\ell \nu_1$.
 We divide $Q_r(x)$ into subcube $Q_j$, $j=1,\ldots,2^{\ell n}$ of side length $2\nu_1$.
 Then
\[
	\frac{1}{\left| Q_r(x) \right|} \int_{Q_r(x)} |\tilde{v}| \, dy 
	\leq \frac{1}{2^{\ell n}(2\nu_1)^n} \sum^{2^{\ell n}}_{j=1} \int_{Q_j} |\tilde{v}| \, dy 
	\leq \frac{2^{\ell n}}{2^{\ell n}(2\nu_1)^n} \| \tilde{v} \|_{L^1_{\mathrm{ul}}} 
	\leq \frac{1}{(2\nu_1)^n} \| \tilde{v} \|_{L^1_{\mathrm{ul}}}
\]
where $r_0$ in $L^1_{\mathrm{ul}}$ norm is taken as $\sqrt{n}\;\nu_1$.
 We thus observe that
\[
	[\nabla d \cdot v]_{b^{\nu_2}} - [\nabla d \cdot v]_{b^{\nu_1}}
	\leq c[v]_{\Gamma_\delta}.
\]
\end{proof}

By Propositions \ref{E1}, \ref{E2} we do not need to care about $\nu$.
 More precisely, 
\begin{theorem} \label{F3}
Let $\Omega$ be an arbitrary domain in $\mathbf{R}^n$.
 Assume that $\mu \in (0,\infty]$ and that $\delta \in (0,\infty]$.
 Then norms $\| \cdot \|_{vbmo^{\mu, \nu_1}_\delta}$ and $\| \cdot \|_{vbmo^{\mu, \nu_2}_\delta}$ are equivalent provided that $0<\nu_1<\nu_2<\infty$.
 In the case $\delta=\infty$, we may take $\nu_2=\infty$. 
\end{theorem}

In general, different from Theorem \ref{B2}, the space $vBMO^{\mu, \nu}$ may not be included in $bmo^\mu_\nu$ even for finite $\mu$ by the decomposition \eqref{DC} and the fact that $BMO^\mu$ is not contained in $L^1_{\mathrm{ul}}(\Gamma_\delta)$ for any $\delta$.
 However, if each connected component of the boundary $\Gamma$ of $\Omega$ has a curved part, we are able to compare these spaces.
\begin{definition} \label{FC}
Let $\Omega$ be a uniformly $C^1$ domain in $\mathbf{R}^n$ and $\Gamma^0$ be a connected component of the boundary $\Gamma$ of $\Omega$.
 We say that $\Gamma^0$ has a \textit{fully curved part} if the set of all normals of $\Gamma^0$ spans $\mathbf{R}^n$.
 In other words, the set $\left\{ \mathbf{n}(x) \in \mathbf{R}^n \mid x \in \Gamma^0 \right\}$ contains $n$ linearly independent vectors, when $\mathbf{n}$ denotes the unit exterior normal of $\Gamma^0$.
\end{definition}

We introduce $b^\nu(\Gamma^0)$-seminorm for convenience. Let us decompose $\Gamma$ into its connected component $\Gamma^j$ so that $\Gamma=\bigcup^m_{j=1}\Gamma^j$.
 We set
\[
	[f]_{b^\nu(\Gamma^j)} := \sup \left\{ r^{-n} \int_{\Omega\cap B_r(x)} \bigl| f(y) \bigr| \, dy \biggm| 
	x \in \Gamma^j,\ 0<r<\nu \right\}.
\]
Evidently, $[f]_{b^\nu} = \max_{1\leq j \leq m} [f]_{b^\nu(\Gamma^j)}$ at least for small $\nu>0$.

The existence of a fully curved part implies ``non-degeneracy'' of the seminorm $[\nabla d \cdot f]_{b^\nu}$.
\begin{lemma} \label{ND}
Let $\Omega$ be a uniformly $C^2$ domain in $\mathbf{R}^n$.
 Let $\Gamma^j$ be a connected component of the boundary $\Gamma$ of $\Omega$.
 If $c \in \mathbf{R}^n$ satisfies
\[
	[\nabla d_\Omega \cdot c]_{b^\nu(\Gamma^j)} = 0,
\]
for some $\nu>0$, then $c=0$ provided that $\Gamma^j$ has a fully curved part.
\end{lemma}
\begin{proof}
If $\Omega$ is uniformly $C^2$, then $d_\Omega$ is $C^2$ in $(\Gamma^j)_\delta$ for sufficiently small $\delta>0$.
 Since $-\nabla d_\Omega(x)$ at $x \in \Gamma^j$ equals $\mathbf{n}(x)$, we see that
\[
	\frac{1}{r^n} \int_{B_r(x)\cap\Omega} \nabla d_\Omega (y) \, dy \to c_0 \mathbf{n}(x)
	\quad\text{as}\quad r \to 0
\]
with scalar constant $c_0$.
 Our assumption implies that $c\cdot\mathbf{n}(x)=0$ for $x \in \Gamma^j$.
 If $\Gamma^j$ has a curved part, then by definition this implies that $c=0$.
\end{proof}

Here is a few comments on examples of such domains.
 All connected components of the boundary of a bounded domain, exterior domain has a fully curved part.
 A perturbed half space
\[
	\mathbf{R}^n_\psi = \left\{ (x', x_n) \in \mathbf{R}^n \bigm|
	x_n > \psi(x'),\ x'=(x_1, \ldots, x_{n-1})\in \mathbf{R}^{n-1} \right\}
\] 
with $\psi \in C^1_c(\mathbf{R}^{n-1})$, $\psi\not\equiv 0$ is another example.
 However, a half-space $\mathbf{R}^n_+$, cylindrical domain $G\times\mathbf{R}^{n-k}$ with $k \geq 1$, $G\subset\mathbf{R}^k$ does not have a boundary having a fully curved part.
 Our goal is to show that for a domain with boundary components having a fully curved part the space $vBMO^{\mu, \nu}$ is comparable with $vbmo^{\mu, \nu}_\delta$ space if the boundary is compact.
\begin{theorem} \label{V1}
Let $\Omega$ be a $C^2$ bounded or exterior domain in $\mathbf{R}^n$ so that each component of the boundary has a fully curved part.
 For $\mu\in (0,\infty]$ and $\nu \in (0,R_*)$ the identity holds:
\[
	vBMO^{\mu, \nu}(\Omega) = vbmo^{\mu, \nu}_\nu.
\]
\end{theorem}
\begin{proof}
Let $\Gamma^j$ be a $j$-th connected component of the boundary $\Gamma=\partial\Omega$ such that $\Gamma=\bigcup^m_{j=1} \Gamma^j$.
 Since $(\Gamma^j)_\nu$ is $C^2$ and compact, there is a number $r_0 \in (0, \nu/2)$ such that 
\[
	(\Gamma^j)_\nu = \bigcup_{x \in \Lambda} \operatorname{int}B_{r_0}(x), \quad
	\Lambda \subset (\Gamma^j)_\nu,
\]
where $\Gamma^j$ is a connected component of $\Gamma$ and $(\Gamma^j)_\nu$ denotes its $\nu$-neighborhood.
 The next lemma shows that
\[
	vBMO^{\nu, \nu}(\Omega) \subset L^1_{\mathrm{ul}}(\Gamma_\nu)
\]
which yields the desired result.
 Note that we may assume $\nu\leq\mu$ by Proposition \ref{E1}.
\end{proof}
\begin{lemma} \label{V2}
Under the same assumption of Theorem \ref{V1} with $\mu\leq\nu$ assume that $r_0 < \nu/2 < R_*/2$ is taken so that
\[
	(\Gamma^i)_\nu = \bigcup_{x \in \Lambda} \operatorname{int}B_{r_0}(x)
\]
with some $\Lambda\subset(\Gamma^j)_\nu$.
 Then there exists $C>0$ depending only on $r_0$, $n$, $\Gamma^j$, $\nu$ such that
\[
	\sup_{x \in \Lambda} \frac{1}{\left| B_{r_0}(x) \right|} \int_{B_{r_0}(x)} \left|f(y)\right| \, dy
	\leq C \Bigl( [f]_{BMO^\mu\left((\Gamma^j)_\nu\right)} + [\nabla d_\Omega \cdot f]_{b^\nu(\Gamma^j)} \Bigr).
\]
\end{lemma}
\begin{proof}
We shall suppress $r_0$ dependence since it is fixed.
 We shall prove the average $\displaystyle f_{B(x)} = \frac{1}{\left| B(x) \right|}\int_{B(x)} f \, dy$ has an estimate
\begin{equation} \label{KI}
	\sup_{x \in \Lambda} \left| f_{B(x)} \right| \leq
	C \Bigl( [f]_{BMO^\mu\left((\Gamma^j)_\nu\right)} + [\nabla d \cdot f]_{b^\nu(\Gamma^j)} \Bigr).
\end{equation}
If this is proved, applying the triangle inequality
\[
\left(|f|\right)_{B(x)} \leq \frac{1}{\left|B(x)\right|} \int_{B(x)} \left| f-f_{B(x)} \right| \, dy + \left| f_{B(x)} \right|
\]
yields the desired result.

We shall prove the key inequality \eqref{KI} by contradiction argument.
 Assume the inequality \eqref{KI} were false.
 Then\tcr{,} there would exist a sequence $\{f^k\}^\infty_{k=1}$ such that
\[
	1 = \sup_{x \in \Lambda} \left| f^k_{B(x)} \right| \geq
	k \left( \left[f^k\right]_{BMO^\mu} + \left[\nabla d_\Omega \cdot f^k\right]_{b^\nu} \right).
\]
Here we suppress $(\Gamma^i)_\nu$ and $\Gamma^j$ in the right-hand side.
 Since
\[
	\sup_{x \in \Lambda} \left| c^k(x) \right| = 1
	\quad\text{with}\quad
	c^k(x) = f^k_{B(x)} \in \mathbf{R}^n,
\]
there is a sequence $\left\{x_k\right\}^\infty_{k=1}$ in $\Lambda$ with the property
\[
	1 \geq \left| c^k(x_k) \right| \geq 1/2.
\]
By taking a subsequence we may assume that $x_k$ converges to some $\hat{x}\in(\Gamma^j)_\nu$ since $\Gamma^j$ is compact and $d\left(x_k, \partial(\Gamma^j)_\nu \right)\geq r_0$, where $d(x_k,A)$ denotes the distance from a point $x_k$ to a set $A$.
 Since $\Gamma^j$ is connected, there is an increasing sequence $\{K_\ell\}^\infty_{\ell=1}$ of connected compact sets in $(\Gamma^j)_\nu$ such that $\operatorname{int}K_\ell \ni\hat{x}$ for $\ell\geq 1$ and $(\Gamma^j)_\nu=\bigcup^\infty_{\ell=1} K_\ell$.
  By compactness, there is a finite subset $\Lambda_\ell$ of $\Lambda$ with the property that
\[
	K_\ell \subset \bigcup_{x\in\Lambda_\ell} \operatorname{int}B(x), \quad
	\Lambda_\ell \subset \Lambda_{\ell+1}
\]
and the right-hand side is connected.
 By taking a further subsequence we may assume that $c^k(x) \to c(x)$ for $x\in\Lambda_\ell$.
 However, since $\left[f^k\right]_{BMO^\mu}\to 0$ so that
\[
	\int_{B(x)} \left| f^k - c^k \right| \, dx \to 0
\]
as $k\to\infty$, we see that $c(x)=c(y)$ if $\operatorname{int}B(x) \cap \operatorname{int}B(y) \neq \emptyset$.
 Since
\[
	\bigcup_{x \in \Lambda_\ell} \operatorname{int} B(x)
\]
is connected, $c(x)$ is independent of $x \in \Lambda_\ell$, say $c=c_\ell$.
 By taking a further subsequence of $\left\{f^k\right\}$ we may assume that $c^k(x)\to c_\ell$ in $\Lambda_\ell$.
 By a diagonal argument there is a subsequence of $\left\{f^k\right\}$ such that
\[
	c^k(x) \to c \quad\text{for}\quad
	x \in \bigcup^\infty_{\ell=1} \Lambda_\ell =: \Lambda_\infty \subset \Lambda.
\]
We thus observe that
\[
	\int_{B(x)} \left| f^k(y) - c \right| \, dy \to 0
	\quad\text{for}\quad x \in \Lambda_\infty
	\quad\text{as}\quad k \to \infty.
\]
If we take $B(x)$ such that $\hat{x} \in \operatorname{int}B(x)$, $c$ should not be equal to zero since $\left| c^k(x_k)\right| \geq 1/2$ and $x_k \to \hat{x}$ as $k \to \infty$.
 We now invoke the property that
\[
	\left[ \nabla d_\Omega \cdot f^k \right]_{b^\nu} \to 0.
\]
Since
\[
	(\Gamma^j)_\nu = \bigcup_{x \in \Lambda_\infty} B(x),
\]
we observe that $f^k \to c$ in $L^1_{\mathrm{loc}}\left((\Gamma^j)_\nu\right)$.
 By taking a subsequence we may assume that $f^k(x) \to c$ for a.e.\ $x \in (\Gamma^j)_\nu$ so that $\nabla d_\Omega \cdot f^k \to \nabla d_\Omega \cdot c$, a.e.
 By lower semicontinuity of integrals (Fatou's lemma) and supremum operation, the seminorm $b^\nu$ is lower semicontinuous under this convergence.
 We thus conclude that
\[
	\left[ \nabla d_\Omega \cdot c \right]_{b^\nu} \leq
	\varliminf_{k\to\infty} \left[ \nabla d_\Omega \cdot f^k \right]_{b^\nu} = 0.
\]
By Lemma \ref{ND} this $c$ must be zero which leads a contradiction.
 We thus proved the key estimate \eqref{KI}.
 This completes the proof of Lemma \ref{V2}.
\end{proof}
%

\section{A variant of Jones' extension theorem} \label{JE} 

Different from $L^\infty$ function it is in general impossible to extend $BMO$ function by setting zero outside the domain.
 Indeed, the zero-extension of $\log\min(x,1) \in bmo^\infty_\infty(\mathbf{R}^1_+)$ does not belong to $BMO^\infty(\mathbf{R})$.
 The goal in this section is to give a linear, extention operator of $BMO$ type function so that the support of extended function is contained in an $\varepsilon$-neighborhood of the original domain, of a function.

For this purpose we recall an extension given by P.\ W.\ Jones \cite{PJ}.
 Since we modify the way of construction, we will give a sketch of this construction.
 We first recall a dyadic Whitney decomposition of a set $A$ in $\mathbf{R}^n$.
 Let $\mathcal{A}=\{Q_j\}_{j\in\mathbf{N}}$ be a set of dyadic closed cubes with side length $\ell(Q_j)$ contained in $A$ satisfying following four conditions.
\begin{enumerate}
\item[(i)] $A = \cup_j Q_j$,  
\item[(i\hspace{-0.1em}i)] $\operatorname{int}Q_j \cap \operatorname{int}Q_k = \emptyset$ if $j=k$,  
\item[(i\hspace{-0.1em}i\hspace{-0.1em}i)] $\sqrt{n} \leq d(Q_j, \mathbf{R}^n \backslash A) / \ell(Q_j) \leq 4\sqrt{n}$ for all $j \in \mathbf{N}$,  
\item[(i\hspace{-0.1em}v)] $1/4 \leq \ell(Q_k) / \ell(Q_j) \leq 4$ if $Q_j \cap Q_k \neq \emptyset$.
\end{enumerate}
We say that $\mathcal{A}$ is called a dyadic Whitney decomposition of $A$.
 Such a decomposition exists for any open sets; see \cite[Chapter V\hspace{-0.1em}I, Theorem 1]{Ste}.
 Here $d(B,C)$ for sets $B,C$ in $\mathbf{R}^n$ is defined as
\[
	d(B,C) = \operatorname{inf} \left\{ |x-y| \bigm| x \in B,\ y \in C \right\}.
\]
If $B$ is a point $x$, we write $d(x,C)$ instead of $d\left( \{x\},C \right)$.

There are at least two important distance functions on $\mathcal{A}$.
 For $Q_j,Q_k \in \mathcal{A}$, a family $\left\{ Q(\ell) \right\}^m_{\ell=0} \subset \mathcal{A}$ is called a Whitney chain of length $m$ if $Q(0)=Q_j$ and $Q(m)=Q_k$ such that $Q(\ell) \cap Q(\ell+1) \neq \emptyset$ for $\ell$ with $0 \leq \ell \leq m-1$.
 Then the length of the shortest Whitney chain connecting $Q_j$ and $Q_k$ gives a distance on $\mathcal{A}$, which is denoted by $d_1(Q_j,Q_k)$.
 The second distance for $Q_j,Q_k \in \mathcal{A}$ is defined as
\[
	d_2(Q_j,Q_k) := \log \left| \frac{\ell(Q_j)}{\ell(Q_k)} \right| + \log \left| \frac{\ell(Q_j,Q_k)}{\ell(Q_j) + \ell(Q_k)} + 1 \right|.
\]
Note that $d_1$ and $d_2$ are invariant under dilation as well as translation and rotation.
 P.\ W.\ Jones \cite{PJ} gives a necessary and sufficient conditon for a domain such that there exists a linear extension operator.
 A domain $\Omega$ is called a uniform domain if there exists constants $a,b>0$ such that for all $x,y \in \Omega$ there exists a rectifiable curve $\gamma \subset \Omega$ of length $s(\gamma) \leq a|x-y|$ with $\min \left\{s \left(\gamma(x,z)\right),\ s\left(\gamma(y,z)\right) \right\} \leq bd(z,\partial\Omega)$, where $\gamma(x,z)$ denotes the part of $\gamma$ between $x$ and $z$ on the curve; see e.g,\ \cite{GO}.
 It is equivalent to say that there is a constant $K>0$ such that
\begin{equation} \label{UD}
	d_1(Q_j,Q_k) \leq K d_2 (Q_j,Q_k)
\end{equation}
for all $Q_j,Q_k \in \mathcal{A}$ and some dyadic Whitney decomposition $\mathcal{A}$ of $\Omega$.
\begin{theorem} \label{PJ} 
Let $A\subset\mathbf{R}^n$ be a uniform domain.
 Then there is a constant $C(K)$ depending only on $K$ in \eqref{UD} such that for each $f \in BMO^\infty(A)$ there is an extension $\overline{f} \in BMO^\infty(\mathbf{R}^n)$ satysfying
\[
	\left[\overline{f}\right]_{BMO^\infty(\mathbf{R}^n)} \leq C(K)[f]_{BMO^\infty(A)}.
\]
The operator $f\mapsto\overline{f}$ is a bounded linear operator.
 Conversely, if there exists such an extension, then $A$ is a uniform domain.
\end{theorem}
A bounded Lipschitz domain is a typical example of a uniform domain.
 The constant $K$ in $\eqref{UD}$ depends only on the Lipschitz regularity of the domain.
 A Lipshitz half space $\mathbf{R}^n_\psi$ is another example of a uniform domain; here $\psi$ is a Lipshitz function on $\mathbf{R}^{n-1}$.

We next note that if we modify the construction by P.\ W.\ Jones, the support of the extension $\overline{f}$ is contained in an $\varepsilon$-neighborhood of $\overline{\Omega}$ if $f$ is also in $L^1_\mathrm{ul}$ type space.
\begin{theorem} \label{PJM}
Let $\Omega\subset\mathbf{R}^n$ be a uniform domain.
For each $\varepsilon>0$ there is a constant $C=C(K,\varepsilon)$ with $K$ in \eqref{UD} such that for each $f \in bmo^\infty_\infty(\Omega)$ there is an extension $\overline{f}\in bmo^\infty_\infty(\Omega_{2\varepsilon})$ such that
\[
	\left[\overline{f}\right]_{bmo^\infty_\infty(\Omega_{2\varepsilon})} \leq
	C[f]_{bmo^\infty_\infty(\Omega)}
\]
and $\operatorname{supp}\overline{f}\subset \overline{\Omega_\varepsilon}$, where
\[
	\Omega_\varepsilon := \left\{ x \in \mathbf{R}^n \bigm|
	d(x, \overline{\Omega}) < \varepsilon \right\}.
\]
The operator $f\mapsto\overline{f}$ is a bounded linear operator.
\end{theorem} 
This can be proved almost along the same way as in \cite{PJ}.
 We shall give an explicit proof.

\begin{proof}
Let $k_{\varepsilon}$ be the smallest integer such that $2^{-k_{\varepsilon}} < \frac{\varepsilon}{5\sqrt{n}}$. So $2^{-k_{\varepsilon}} \geq \frac{\varepsilon}{10\sqrt{n}}$. Let $E = \{ Q_j \}$ be the Whitney decomposition of $\Omega$ and $E^{'} = \{ Q_j^{'} \}$ be the Whitney decomposition of $\Omega^{c}$. Let $E_\ast$ be the set of Whitney cubes in $E$ whose side length is strictly greater than $2^{-k_{\varepsilon}}$. For each $Q_m \in E_\ast$, we define a function $g_m$ on $\Omega$ by
\begin{eqnarray*}
	g_m(x) := \left \{
\begin{array}{ll}
	f_{Q_m}, & \mathrm{if} \; \, x \in Q_m \\
	0, & \mathrm{else}
\end{array}
\right.
\end{eqnarray*}
and we further define a function $g$ on $\Omega$ by 
\[
g := \sum_{Q_m \in E_\ast} g_m.
\]
Here $f_{Q_m} = \frac{1}{|Q_m|} \int_{Q_m} f(y) dy$ for each $Q_m \in E_\ast$. Let $\widetilde{g}$ be the zero extension of $g$ from $\Omega$ to $\mathbf{R}^n$. 

Without loss of generality, we assume that the radius $r_0$ of the ball equals $1$ in defining the space $L_{\mathrm{ul}}^1(\Omega)$. Notice that
\[
\| g_m \|_{L^\infty(\Omega)} \leq \frac{1}{|Q_m|} \cdot \int_{Q_m} |f| \, dy.
\]
Let $k_0$ be the smallest integer such that $2^{-k_0} < \frac{2}{\sqrt{n}}$. If $\ell(Q_m) \leq 2^{-k_0}$, then $\| f \|_{L^1(Q_m)} \leq [f]_{\Gamma_{\infty}}$. In this case, as $\ell(Q_m) > 2^{-k_\varepsilon}$, 
\[
\| g_m \|_{L^\infty(\Omega)} \leq \frac{1}{|Q_m|} \cdot \int_{Q_m} |f| \, dy \leq (\frac{10\sqrt{n}}{\varepsilon})^n \cdot [f]_{\Gamma_{\infty}}.
\]
If $\ell(Q_m) > 2^{-k_0}$, we divide $Q_m$ into $(\frac{\ell(Q_m)}{2^{-k_0}})^n$ small subcubes of side length $2^{-k_0}$. Hence,
\[
\int_{Q_m} |f| \, dy = \sum_{i=1}^{(\ell(Q_m)/2^{-k_0})^n} \int_{Q_m^i} |f| \, dy \leq (\frac{\ell(Q_m)}{2^{-k_0}})^n \cdot [f]_{\Gamma_{\infty}} \leq |Q_m| \cdot n^{\frac{n}{2}} \cdot [f]_{\Gamma_{\infty}},
\]
in this case $\| g_m \|_{L^\infty(\Omega)} \leq n^{\frac{n}{2}} \cdot [f]_{\Gamma_{\infty}}$. Therefore,
\[
\| g \|_{L^\infty(\Omega)} \leq C_{n,\varepsilon} \cdot [f]_{\Gamma_{\infty}}
\]
and we deduce that $g \in bmo_\infty^\infty(\Omega)$ as $L^\infty(\Omega) \subset bmo_\infty^\infty(\Omega)$.

Let $f^\ast := f -g \in bmo_\infty^\infty(\Omega)$. We do Jones extension to $f^\ast$. If $\Omega$ is unbounded, for each $Q_j^{'} \in E^{'}$, we find a nearset $Q_j \in E$ satisfying $\ell(Q_j) \geq \ell(Q_j^{'})$. We define that $\widetilde{f}^\ast = f^\ast$ on $\Omega$ and $\widetilde{f}^\ast(x) = f^\ast_{Q_j}$ for $x \in Q_j^{'}$. If $\Omega$ is bounded, we pick $Q_0 \in E$ such that $\ell(Q_0) = \underset{Q_j \in E}{\mathrm{sup}} \ell(Q_j)$. We define that $\widetilde{f}^\ast = f^\ast$ on $\Omega$, $\widetilde{f}^\ast(x) = f^\ast_{Q_j}$ for $x \in Q_j^{'}$ where $\ell(Q_j^{'}) \leq \ell(Q_0)$ and $\widetilde{f}^\ast(x) = f^\ast_{Q_0}$ for $x \in Q_j^{'}$ where $\ell(Q_j) > \ell(Q_0)$. By Jones \cite{PJ}, $\widetilde{f}^{\ast} \in BMO$ and $[\widetilde{f}^{\ast}]_{BMO} \leq C_K \cdot [f^\ast]_{BMO^\infty(\Omega)}$. By this extension, for $\widetilde{f}^\ast(x) \neq 0$, either $x \in \Omega$ or $x \in Q_j^{'}$ such that $\ell(Q_j^{'}) \leq 2^{-k_{\varepsilon}}$. Since $d(Q_j^{'},\Omega) \leq 4\sqrt{n} \cdot \ell(Q_j^{'})$, pick $x \in \overline{Q_j^{'}}$ and $z \in \Gamma$ such that $|x -z |= d(Q_j^{'},\Omega)$. For any $y \in Q_j^{'}$, $|y-z| \leq |y-x|+|x-z| \leq 5\sqrt{n} \cdot \ell(Q_j^{'})$. So $\mathrm{int} \, Q_j^{'} \subset B_{5\sqrt{n} \cdot \ell(Q_j^{'})}(z)$ for some $z \in \Gamma$. Since $5\sqrt{n} \cdot \ell(Q_j^{'}) \leq 5\sqrt{n} \cdot 2^{-k_{\varepsilon}} < \varepsilon$, $\mathrm{int} \, Q_j^{'} \subset \Omega_{\varepsilon}$. Let $\widetilde{f} := \widetilde{f}^\ast + \widetilde{g}$ and $\overline{f} = \widetilde{f} \, |_{\Omega_{2\varepsilon}}$, we have that $\mathrm{supp} \; \overline{f} \subset \overline{\Omega_{\varepsilon}}$ and by previous calculation, 
\begin{eqnarray*}
[\overline{f}]_{BMO^\infty(\Omega_{2\varepsilon})} &\leq& [\widetilde{f}]_{BMO} \leq [\widetilde{f}^\ast]_{BMO} + [\widetilde{g}]_{BMO} \leq C_K \cdot [f^\ast]_{BMO^\infty(\Omega)} + 2||g||_{\infty} \\
&\leq& C_{K,n,\varepsilon} \cdot ([f]_{BMO^\infty(\Omega)} + [f]_{\Gamma_{\infty}}).
\end{eqnarray*}

Let $B(x)$ denotes the ball of radius $1$ centered at $x$ and $\Gamma^\varepsilon := \{ x \in \Omega^{c} \, | \, d_\Omega(x) < \varepsilon \}$. For $B(x) \cap \Omega_\varepsilon \neq \emptyset$,
\[
\int_{B(x) \cap \Omega_\varepsilon} |\overline{f}| \, dy = \int_{B(x) \cap \Omega} |f| \, dy + \int_{B(x) \cap \Gamma^\varepsilon} |\overline{f}| \, dy.
\]
The first integral on the right hand side is directly estimated by $[f]_\infty$, so we only need to consider the second integral. Let $Q_\ast^{'}$ be a largest Whitney cube in $E^{'}$ that intersects $B(x) \cap \Gamma^\varepsilon$. For $Q_j^{'} \in E^{'}$, \cite[Lemma 2.10]{PJ} says that if $Q_j \in E$ is a nearest Whitney cube satisfying $\ell(Q_j) \geq \ell(Q_j^{'})$, then $d(Q_j,Q_j^{'}) \leq 65K^2 \cdot \ell(Q_j^{'})$. Consider $Q_j^{'} \in E^{'}$ such that $Q_j^{'} \cap B(x) \cap \Gamma^\varepsilon \neq \emptyset$, let $x_j \in Q_j$ where $Q_j$ is a nearest Whitney cube satisfying $\ell(Q_j) \geq \ell(Q_j^{'})$, let $x_j^{'} \in Q_j^{'} \cap B(x) \cap \Gamma^\varepsilon$ and $x_\ast^{'} \in Q_\ast^{'} \cap B(x) \cap \Gamma^\varepsilon$. By choosing $K$ large such that $K^2 \geq 2\sqrt{n}$, we have that
\[
|x_j - x_\ast^{'}| \leq |x_\ast^{'} - x_j^{'}| + |x_j^{'} - x_j| \leq 2 + 2\sqrt{n} \cdot \ell(Q_j^{'}) + 65K^2 \cdot \ell(Q_j^{'}) \leq 2 + 66K^2 \cdot \ell(Q_j).
\]
Since $\ell(Q_j) \leq 2\ell(Q_j^{'}) \leq 2\ell(Q_\ast^{'}) \leq 2\ell(Q_\ast)$ where $Q_\ast \in E$ is a nearest cube satisfying $\ell(Q_\ast) \geq \ell(Q_\ast^{'})$, $|x_j - x_\ast^{'}| \leq 2 + 132K^2 \cdot \ell(Q_\ast)$.

If $B(x) \cap \Gamma \neq \emptyset$, then $\sqrt{n} \cdot \ell(Q_\ast^{'}) \leq d(Q_\ast^{'},\Omega) \leq 2$. Hence $\ell(Q_\ast) \leq 2\ell(Q_\ast^{'}) \leq \frac{4}{\sqrt{n}}$, for any $x_j \in Q_j$, $|x_j - x_\ast^{'}| < 2 + 133K^2 \cdot \frac{4}{\sqrt{n}}$. Consider the cube $\widetilde{Q_\ast^{'}}$ with center $x_\ast^{'}$ and side length $4+ \frac{1064K^2}{\sqrt{n}}$. For each $Q_j^{'} \in E^{'}$ such that $Q_j^{'} \cap B(x) \cap \Gamma^\varepsilon \neq \emptyset$, the corresponding nearest $Q_j \in E$ such that $\ell(Q_j) \geq \ell(Q_j^{'})$ we choose to define $\widetilde{f}^\ast$ is contained in $\widetilde{Q_\ast^{'}}$, i.e. $Q_j \subset \widetilde{Q_\ast^{'}}$. Hence,
\[
\int_{B(x) \cap \Gamma^\varepsilon} |\overline{f}| \, dy = \underset{Q_j^{'} \cap B(x) \cap \Gamma^\varepsilon \neq \emptyset}{\underset{Q_j^{'} \in E^{'},}{\sum}} \int_{Q_j^{'} \cap B(x) \cap \Gamma^\varepsilon} |f_{Q_j}^\ast| \, dy \leq \int_{\widetilde{Q_\ast^{'}} \cap \Omega} |f^\ast| \, dy.
\]
Let $p$ be the largest integer such that $2^{-p} > 4 + \frac{1064K^2}{\sqrt{n}}$, so $2^{-p} \leq 8+ \frac{2128K^2}{\sqrt{n}}$. Let $\widetilde{Q_\ast^{'}}$ be contained in a larger cube $\widetilde{Q}$ where $\widetilde{Q}$ has center $x_\ast^{'}$ and side length $2^{-p}$. We can divide $\widetilde{Q}$ into $(\frac{2^{-p}}{2^{-k_0}})^n$ subcubes of side length $2^{-k_0}$, thus
\[
\int_{\widetilde{Q_\ast^{'}} \cap \Omega} |f^\ast| \, dy \leq \sum_{i=1}^{(2^{-p}/2^{-k_0})^n} \int_{\widetilde{Q_i} \cap \Omega} |f^\ast| \, dy \leq (\frac{2^{-p}}{2^{-k_0}})^n \cdot [f^\ast]_{\Gamma_{\infty}} \leq C_{K,n} \cdot [f^\ast]_{\Gamma_{\infty}}.
\]

If $B(x) \cap \Gamma = \emptyset$, i.e. $B(x) \subset \overline{\Omega}^{c}$. Let $E_1^{'} := \{ Q_j^{'} \in E^{'} \, | \, Q_j^{'} \cap B(x) \neq \emptyset \}$. Let $\ell_m := \underset{Q_j^{'} \in E_1^{'}}{\mathrm{inf}} \ell(Q_j^{'})$ and $Q_\ast^{'}$ be a largest $Q_j^{'} \in E_1^{'}$. If $\ell_m = 0$, then there exists $z \in \Gamma \cap \partial B(x)$. In this case, $\sqrt{n} \cdot \ell(Q_\ast^{'}) \leq d(Q_\ast^{'},\Omega) \leq 2$. Therefore same argument as in the case where $B(x) \cap \Gamma \neq \emptyset$ gives that $\| \overline{f} \|_{L^1(B(x) \cap \Gamma^\varepsilon)} \leq C_{K,n} \cdot [f^\ast]_\infty$. If $0<\ell_m \leq 2$, then pick $Q_m^{'} \in E_1^{'}$ such that $\ell(Q_m^{'}) = \ell_m$. Since $\sqrt{n} \cdot \ell(Q_\ast^{'}) \leq d(Q_\ast^{'},\Omega) \leq 2+ \sqrt{n} \cdot \ell(Q_m^{'}) + d(Q_m^{'},\Omega) \leq 2+ 10\sqrt{n}$, we have that $\ell(Q_\ast) \leq \frac{4}{\sqrt{n}} + 20$. Hence $|x_j - x_\ast^{'}| \leq 2+ 133K^2 \cdot (\frac{4}{\sqrt{n}} + 20)$. Following the argument as in the case where $B(x) \cap \Gamma \neq \emptyset$, we can deduce that $\| \overline{f} \|_{L^1(B(x) \cap \Gamma^\varepsilon)} \leq C_{K,n} \cdot [f^\ast]_{\Gamma_{\infty}}$. If $\ell_m >2$, then $B(x)$ intersects at most $2^n$ Whitney cubes in $E^{'}$. Without loss of generality, assume that $E_1^{'}$ has $2^n$ elements. Then
\[
\int_{B(x) \cap \Gamma^\varepsilon} |\overline{f}| \, dy \leq \sum_{Q_i^{'} \in E_1^{'}} \int_{B(x) \cap Q_i^{'}} |f_{Q_i}^\ast| \, dy \leq \sum_{Q_i^{'} \in E_1^{'}} \frac{|B(x) \cap Q_i^{'}|}{|Q_i|} \cdot \int_{Q_i} |f^\ast| \, dy.
\]
Divide $Q_i$ into $\left(\frac{\ell(Q_i)}{2^{-k_0}}\right)^n$ subcubes of side length $2^{-k_0}$, we have that
\[
\int_{Q_i} |f^\ast| \,dy \leq \left(\frac{\ell(Q_i)}{2^{-k_0}}\right)^n \cdot [f^\ast]_{\Gamma_{\infty}} \leq |Q_i| \cdot n^\frac{n}{2} \cdot [f^\ast]_{\Gamma_{\infty}}.
\]
Therefore,
\[
\int_{B(x) \cap \Gamma^\varepsilon} |\overline{f}| \, dy \leq \big( \sum_{Q_i^{'} \in E_1^{'}} |B(x) \cap Q_i^{'}| \,\big) \cdot n^\frac{n}{2} \cdot [f^\ast]_{\Gamma_{\infty}} \leq C_n \cdot [f^\ast]_{\Gamma_{\infty}}.
\]
Since $[f^\ast]_{\Gamma_{\infty}} \leq [f]_{\Gamma_{\infty}} + [g]_{\Gamma_{\infty}}$ and $[g]_{\Gamma_{\infty}}$ is estimated by $C_{n,\varepsilon} \cdot [f]_{\Gamma_{\infty}}$, we are done.
\end{proof}

As an application we give an estimate for the product of a H$\ddot{\mathrm{o}}$lder function and a function in $bmo^\infty_\infty$.
 We first recall properties of point multipliers.
 It is known that for a local hardy space $h^1=F^0_{1,2}$ \cite[Theorem 3.18]{Sa}, there is a constant $C$ such that
\begin{align} \label{MHA}
	\| \varphi g \|_{F^0_{1,2}} \leq 
	C\| \varphi \|_{C^\gamma} \|g\|_{F^0_{1,2}} \quad
	g \in F^0_{1,2}
\end{align}
for $\varphi \in C^\gamma(\mathbf{R}^n)$, $\gamma \in (0,1)$, where
\[
	\| \varphi \|_{C^\gamma} = \sup_{x\in\mathbf{R}^n} \left|\varphi(x)\right| + \sup_{\substack{x,y\in\mathbf{R}^n \\ x \neq y}} \left|\varphi(x)-\varphi(y)\right| \bigm/ |x-y|^\gamma; 
\]
see e.g.\ \cite[Remark 4.4]{Sa}.
 Since
\[
	bmo = BMO^\infty(\mathbf{R}^n) \cap L^1_\mathrm{ul}(\mathbf{R}^n)
\]
equals to $F^0_{\infty,2}$ \cite[Theorem 3.26]{Sa}, it is a dual space of $h^1=F^0_{1,2}$ \cite[Theorem 3.22]{Sa}.
 Thus
\begin{equation} \label{ME}
	\| \varphi f \|_{bmo} \leq C \| \varphi \|_{C^\gamma} \| f \|_{bmo}.
\end{equation}
\begin{theorem} \label{PMU}
Let $\Omega \subset \mathbf{R}^n$ be a uniform domain. Let $\varphi \in C^\gamma(\Omega)$, $\gamma \in (0,1)$. For each $f \in bmo^\infty_\infty(\Omega)$, the function $\varphi f \in bmo^\infty_\infty(\Omega)$ satisfies
\[
	\| \varphi f \|_{bmo^\infty_\infty(\Omega)} \leq C \| \varphi \|_{C^\gamma(\Omega)} \| f \|_{bmo^\infty_\infty(\Omega)}
\]
with $C$ independent of $\varphi$ and $f$.
\end{theorem} 
\begin{proof}
By \cite{McS}, there exists $\overline{\varphi} \in C^{\gamma}(\mathbf{R}^n)$ such that $\overline{\varphi} \, |_{\Omega} = \varphi$ and 
\[
\| \overline{\varphi} \|_{C^{\gamma}(\mathbf{R}^n)} \leq \| \varphi \|_{C^{\gamma}(\Omega)}.
\]
For our current purpose it suffices to set $\overline{\varphi} = \mathrm{max} \{ \mathrm{min} \{ \varphi_\ast, \| \varphi \|_\infty \}, - \| \varphi \|_\infty \}$ with
\[
\varphi_\ast(x) = \underset{y \in \Omega}{\mathrm{inf}} \left\{ \varphi(y) + [\varphi]_{C^\gamma} \cdot |x-y|^\gamma \right\},
\]
where $\| \varphi \|_{C^\gamma(\Omega)} = \| \varphi \|_{L^\infty(\Omega)} + [\varphi]_{C^\gamma(\Omega)}$, $\| \varphi \|_{L^\infty(\Omega)} = \underset{x \in \Omega}{\mathrm{sup}} | \varphi (x) |$ and $[\varphi]_{C^\gamma(\Omega)} = \underset{x,y \in \Omega}{\mathrm{sup}} \frac{|\varphi(x) - \varphi(y)|}{|x-y|^\gamma}$; we often suppress $\Omega$. By definition $\varphi_\ast(x) \leq \varphi(x)$. Moreover, since $\varphi(x) \leq \varphi(y)+[\varphi]_{C^\gamma} \cdot |x-y|^\gamma$ for $x,y \in \Omega$, we see that $\varphi(x) \leq \varphi_\ast(x)$ which implies $\varphi = \varphi_\ast$ on $\Omega$. For any $x \in \mathbf{R}^n$ and $\varepsilon > 0$ there is $y_\varepsilon \in \Omega$ such that
\[
\varphi(y_\varepsilon) + [\varphi]_{C^\gamma} \cdot |x-y_\varepsilon|^\gamma \leq \varphi_\ast(x) + \varepsilon.
\]
For $x_1 \in \mathbf{R}^n$ we observe that
\[
\varphi_\ast(x_1) - \varphi_\ast(x) \leq \varphi(y_\varepsilon) + [\varphi]_{C^\gamma} \cdot |x_1-y_\varepsilon|^\gamma - \left\{ \varphi(y_\varepsilon) + [\varphi]_{C^\gamma} \cdot |x-y_\varepsilon|^\gamma \right\} + \varepsilon \leq [\varphi]_{C^\gamma} \cdot |x-x_1|^\gamma + \varepsilon.
\]
Since $\varepsilon$ is arbitrary, we see that $\varphi_\ast(x_1) - \varphi_\ast(x) \leq [\varphi]_{C^\gamma} \cdot |x-x_1|^\gamma$. Interchanging the role of $x_1$ and $x$, we conclude that
\[
[\varphi_\ast]_{C^\gamma(\mathbf{R}^n)} \leq [\varphi]_{C^\gamma(\Omega)}.
\]
Since $\| \varphi \|_\infty < \infty$, $\overline{\varphi} = \varphi$ on $\Omega$ and $\overline{\varphi}$ is still H$\ddot{\mathrm{o}}$lder. More precisely, $[\overline{\varphi}]_{C^\gamma} \leq [\varphi_\ast]_{C^\gamma}$. By definition $\| \overline{\varphi} \|_\infty \leq \| \varphi \|_\infty$ so we conclude that $\| \overline{\varphi} \|_{C^\gamma} \leq \| \varphi \|_{C^\gamma}$.

Extending $f \in bmo_\infty^\infty(\Omega)$ to $\overline{f} \in bmo$ by Theorem \ref{PJM}, we conclude from multiplication estimate (\ref{ME}) that
\begin{eqnarray*}
\| \varphi f \|_{bmo_\infty^\infty(\Omega)} &\leq& \| \overline{\varphi} \overline{f} \|_{bmo} \\
&\leq& C \cdot \| \overline{\varphi} \|_{C^{\gamma}(\mathbf{R}^n)} \cdot \| \overline{f} \|_{bmo} \\
&\leq& C \cdot \| \varphi \|_{C^{\gamma}(\Omega)} \cdot \| f \|_{bmo_\infty^\infty(\Omega)}.
\end{eqnarray*}
\end{proof}

\begin{remark} \label{POZ}
If we prove that the extension $f \mapsto \overline{f}$ constructed in Theorem \ref{PJ} is bounded from $bmo_\infty^\infty$ to $bmo = BMO \cap L_{\mathrm{ul}}^1$, then the support condition will follow by taking $\varphi \in C^\gamma(\mathbf{R}^n)$ in Theorem \ref{PMU} as a cutoff function of $\Omega$, i.e. $\varphi \equiv 1$ on $\Omega$ with $\mathrm{supp} \varphi \subset \Omega_\varepsilon$. In other words, we consider $f \mapsto \varphi \overline{f}$. However, the proof that $\overline{f} \in L_{\mathrm{ul}}^1$ needs some argument so we give a direct proof of Theorem \ref{PJM}.
\end{remark}

For $BMO_b^{\mu,\infty}$ function in $\Omega$ it is easy to see that its zero extension is in $BMO$ space; see e.g. \cite[Lemma 4]{BGST}.

\begin{theorem} \label{ZE}
Let $\Omega$ be an arbitrary domain in $\mathbf{R}^n$. Assume that $\mu \in (0,\infty]$. For $f \in BMO_b^{\mu,\nu}(\Omega)$ with $\nu \geq 2\mu$, let $f_0$ be the zero extension to $\mathbf{R}^n$, i.e. $f_0(x) = 0$ for $x \in \Omega^c$ and $f_0(x) = f(x)$ for $x \in \Omega$. Then $f_0 \in BMO^\mu(\mathbf{R}^n)$ and $[f_0]_{BMO^\mu} \leq C [f]_{BMO_b^{\mu,\nu}}$ with $C$ independent of $f$.
\end{theorem}

\begin{proof}
If the ball $B$ of radius $\leq \mu$ is in $\Omega$, then
\[
\frac{1}{|B|} \int_B | f_0 - {f_0}_B | \, dy \leq [f]_{BMO^\mu}.
\]
If $B$ is in $\Omega^c$, then $\int_B | f_0 - {f_0}_B | \, dy = 0$. It remains to estimate the integral if $B$ has nonempty intersection with the boundary $\Gamma = \partial \Omega$. For each $B_r(x) \cap \Gamma \neq \emptyset$, $r < \mu$, we take $x_0 \in B_r(x) \cap \Gamma$. Then, $B_r(x) \subset B_{2r}(x_0)$ and thus
\[
\frac{1}{|B_r(x)|} \int_{B_r(x)} | f_0 - {f_0}_{B_r(x)} | \,dy \leq \frac{2}{|B_r(x)|} \int_{B_{2r}(x_0)} |f_0| \, dy \leq \frac{2^{n+1}}{\omega_n} \cdot [f]_{b^{2\mu}},
\]
where $\omega_n$ is the volume of an $n$-dimensional ball.
\end{proof}

\begin{remark} \label{BGST}
In \cite[Lemma 4]{BGST}, it is assumed that $\Omega = \Omega^{'} \times \mathbf{R}^{n-k}$ where $\Omega^{'}$ is a bounded Lipschitz domain in $\mathbf{R}^k$. However, from the proof above it is clear that we do not need this requirement. Thus we give a full proof here.
\end{remark}

As an application of boundedness of multiplication, we give invariance of function spaces under coordinate changes. We say that $\Psi$ is a global $C^{k+\beta}$ (resp. $C^k$)-diffeomorphism if $C^{k+\beta}$ (resp. $C^k$)-norm of $\Psi$ and $\Psi^{-1}$ are bounded in $\mathbf{R}^n$, where $k \in \mathbf{N}$ and $\beta \in (0,1)$.

\begin{proposition} \label{CC}
The space $bmo$ is invariant under bi-Lipschitz coordinate change and the space $h^1$ is invariant under global $C^{1+\beta}$-diffeomorphism.
\end{proposition}

\begin{proof} 
For $f \in bmo$, by a simple change of variables on the equivalent definition of the seminorm $[f]_{BMO}$ where
\[
[f]_{BMO} = \underset{B \subset \mathbf{R}^n}{\mathrm{sup}} \; \underset{c \in \mathbf{R}}{\mathrm{inf}} \, \int_B |f(y)-c| \, dy,
\]
see e.g. \cite[Proposition 3.1.2]{Gra}, we can easily deduce that $bmo$ is invariant under bi-Lipschitz coordinate change.

Let $g \in h^1(\mathbf{R}^n)$ and $\Psi$ be a global $C^{1+\beta}$-diffeomorphism. We have that
\[
\| g \circ \Psi \|_{h^1} = \underset{ \| f \|_{bmo} \leq 1}{\mathrm{sup}} \, \left| \int_{\mathbf{R}^n} f \cdot g \circ \Psi \, dy \right|.
\]
By change of variable we have that
\[
\left| \int_{\mathbf{R}^n} f(y) \cdot g \circ \Psi (y) \, dy \right| = \left| \int_{\mathbf{R}^n} f \circ \Psi^{-1} (x) \cdot g(x) \cdot J_{\Psi^{-1}}(x) \, dx \right|\tcr{,}
\]
where $J_{\Psi^{-1}}$ is the Jacobian which is of regularity $C^{\beta}$. Then by the $bmo-h^1$ duality \cite[Theorem 3.22]{Sa} and multiplication estimate (\ref{MHA}), we deduce that
\[
\left| \int_{\mathbf{R}^n} f \circ \Psi^{-1} \cdot g \cdot J_{\Psi^{-1}} \, dx \right| \leq \| f \circ \Psi^{-1} \|_{bmo} \cdot \| g J_{\Psi^{-1}} \|_{h^1} \leq \| f \circ \Psi^{-1} \|_{bmo} \cdot \| J_{\Psi^{-1}} \|_{C^\beta} \cdot \| g \|_{h^1}.
\]
Since $bmo$ is independent of bi-Lipschitz coordinate change, we have that
\[
\| g \circ \Psi \|_{h^1} \leq C \cdot \| \nabla \Psi^{-1} \|_{L^\infty} \cdot \| J_{\Psi^{-1}} \|_{C^\beta} \cdot \| g \|_{h^1}
\]
for some constant $C$ independent of $g$ and $\Psi$.
\end{proof}

\begin{proposition} \label{CCL}
The space $F_{1,2}^1(\mathbf{R}^n)$ is independent of global $C^{1+\beta}$-diffeomorphism.
\end{proposition}

\begin{proof}
Let $g \in F_{1,2}^1$ and $\Psi$ be a global $C^{1+\beta}$-diffeomorphism. By multiplication estimate (\ref{MHA}) and Proposition \ref{CC}, we have that
\[
\| \nabla (g \circ \Psi) \|_{F_{1,2}^0} \leq C \cdot \| \nabla \Psi \|_{C^\beta} \cdot \| (\nabla g) \circ \Psi \|_{F_{1,2}^0} \leq C \cdot \| \nabla \Psi \|_{C^\beta} \cdot \| \nabla \Psi^{-1} \|_{L^\infty} \cdot \| J_{\Psi^{-1}} \|_{C^\beta} \cdot \| g \|_{F_{1,2}^1}\tcr{,}
\]
where $J_{\Psi^{-1}}$ is the Jacobian for $\Psi^{-1}$ and $C$ is a constant independent of $g$ and $\Psi$. Hence $\nabla (g \circ \Psi) \in F_{1,2}^0$. Since the differentiation mapping is bounded from $F_{p,q}^s$ to $F_{p,q}^{s-1}$ for $p \in (0,\infty)$, $q \in (0,\infty]$ and $s \in \mathbf{R}$, see e.g. \cite[Theorem 2.12]{Sa}, we have that $\Delta (g \circ \Psi) \in F_{1,2}^{-1}$. Since $F_{1,2}^1 \hookrightarrow F_{1,2}^0$, Proposition \ref{CC} tells us that $g \circ \Psi \in F_{1,2}^0$ and thus $g \circ \Psi \in F_{1,2}^{-1}$. Therefore, $(I - \Delta) (g \circ \Psi) \in F_{1,2}^{-1}$. Notice that \cite[Theorem 2.12]{Sa} also tells us that for $\sigma \in \mathbf{R}$, $(I - \Delta)^{\sigma}$ is an isomorphism from $F_{p,q}^s$ to $F_{p,q}^{s-2\sigma}$. Hence by letting $\sigma = -1$, we deduce that 
\begin{align*}
\| g \circ \Psi \|_{F_{1,2}^1} &= \| (I - \Delta)^{-1} (I - \Delta) (g \circ \Psi) \|_{F_{1,2}^1} \\
&= C \cdot \| (I - \Delta) (g \circ \Psi) \|_{F_{1,2}^{-1}} \\
&\leq C \cdot \left( \| g \circ \Psi \|_{F_{1,2}^0} + \| \nabla (g \circ \Psi) \|_{F_{1,2}^0} \right)\\
&\leq C \cdot (1+ \| \nabla \Psi \|_{C^\beta}) \cdot \| \nabla \Psi^{-1} \|_{L^\infty} \cdot \| J_{\Psi^{-1}} \|_{C^{\beta}} \cdot \| g \|_{F_{1,2}^1},
\end{align*}
where $C$ is a constant independent of $g$ and $\Psi$.
\end{proof}

\begin{remark} \label{CCLR}
The proof of Proposition \ref{CCL} also says that $F_{1,2}^1 = \{ f \in F_{1,2}^0 \mid \nabla f \in (F_{1,2}^0)^n \}$.
\end{remark}
%

\section{Trace problems} \label{TR} 

In this section we show that the normal trace of a vector field in $vbmo^{\mu, \nu}_\delta$ is in $L^\infty(\Gamma)$ if its divergence is well controlled.
 We begin with the case that $\Omega$ is the half space $\mathbf{R}^n_+$.

We first recall that the trace operator $(Tr f) (x') = f(x',0)$ for $f \in F^1_{1, 2}(\mathbf{R}^n)$ gives a surjective bounded linear operator from $F^1_{1, 2}(\mathbf{R}^n)$ to $L^1(\mathbf{R}^{n-1})$; see \cite[Section 4.4.3]{Tr92}.
\begin{proposition}[\cite{Tr92}] \label{HA}
The operator $Tr$ from $F^1_{1, 2}$ to $L^1(\mathbf{R}^{n-1})$ is surjective for $n \geq 2$.
 Actually, surjectivity holds for a smaller space $B^1_{1, 1}$.
 The inverse operator is called the extension and it is a bounded operator.
\end{proposition}
For a $C^2$ domain $\Omega$ a normal trace $v\cdot\mathbf{n}$ on $\Gamma=\partial\Omega$ of $v$ is well-defined as an element of $W^{-1/p}_{p, \mathrm{loc}}(\Gamma)$ if $v$ and $\operatorname{div}v$ is in $L^p_\mathrm{loc}$; see e.g.\ \cite{FM} or \cite{Gal}.
 If $v \in vbmo^{\mu, \nu}_\delta(\Omega)$ so that $v \in L^1_\mathrm{loc}$, then by an interpolation inequality (see e.g.\ \cite[Theorem 11]{BGST}) $v$ is in $L^p_\mathrm{loc}$ for any $p \geq 1$.
 Thus if $\operatorname{div}v$ is in $L^p_\mathrm{loc}$, $v\cdot\mathbf{n}$ is well-defined.
 We derive $L^\infty$ estimate for $v\cdot\mathbf{n}$ when $\Omega$ is the half space.
\begin{theorem} \label{NTH}
Let $\mu,\nu,\delta$ be in $(0,\infty]$ and $n \geq 2$.
 Then there is a constant $C=C(\mu,\nu,\delta,n)$ such that
\[
	\| v\cdot\mathbf{n} \|_{L^\infty(\mathbf{R}^{n-1})} 
	\leq C \left( \| v \|_{vbmo^{\mu,\nu}_\delta(\mathbf{R}^n_+)} + \|\operatorname{div}v \|_{L^n_\mathrm{ul}(\Gamma_\delta)} \right)
\]
for all $v \in vbmo^{\mu,\nu}_\delta(\mathbf{R}^n_+)$.
\end{theorem}
\begin{proof}
Let $v \in vbmo^{\mu,\nu}_\delta(\mathbf{R}_{+}^{n})$, by definition the $n$-th component $v_n$ of $v=(v',v_n)$ belongs to $BMO^{\mu,\nu}_b(\mathbf{R}^n_+)$. For $x_0^{'} \in \mathbf{R}^{n-1}$, we consider the region $U = B_1(x_0^{'}) \times (-\delta,\delta)$ where $B_1(x_0^{'})$ denotes the ball in $\mathbf{R}^{n-1}$ centered at $x_0^{'}$ with radius $1$. Let $v_{\mathrm{re}}$ denotes the restriction of $v$ on $U \cap \mathbf{R}_{+}^{n}$, i.e. $v_{\mathrm{re}} = v \, |_{U \cap \mathbf{R}_{+}^{n}}$. We have that $v_{\mathrm{re}} \in bmo_\infty^\infty(U \cap \mathbf{R}_{+}^{n})$ and
\[
\underset{r < \nu}{\underset{x^{'} \in B_1(x_0^{'})}{\mathrm{sup}}} \, \frac{1}{|B_r((x^{'},0))|} \int_{B_r((x^{'},0))} |(v_{\mathrm{re}})_n | \, dy < \infty.
\]
Let $\overline{(v_{\mathrm{re}})_n}$ be the zero extension of $(v_{\mathrm{re}})_n$ to $U$. By Theorem \ref{ZE}, $\overline{(v_{re})_n}$ is in $BMO^\infty(U)$. Let $\overline{v_{\mathrm{re}}^{'}}$ be the even extension of $v_{\mathrm{re}}^{'}$ to $U$ of the form
\begin{eqnarray}
	\overline{v_{\mathrm{re}}^{'}}(x', x_n) = \left \{
\begin{array}{ll}
	v_{\mathrm{re}}'(x', x_n), & x^{'} \in B_1(x_0^{'}) \text{ and } x_n > 0 \\
	v_{\mathrm{re}}'(x', -x_n), & x^{'} \in B_1(x_0^{'}) \text{ and } x_n < 0
\end{array}
\right.
\end{eqnarray}
and set $\widetilde{v} = (\overline{v_{\mathrm{re}}^{'}},\overline{(v_{\mathrm{re}})_n})$. We have that $\widetilde{v} \in bmo_\infty^\infty(U)$. By Theorem \ref{PJM} its Jones' extension $v_U$ belongs to $bmo^\infty_\infty(\mathbf{R}^n)$.

Integration by parts formally yields
\begin{equation} \label{INT}
	\int_{\mathbf{R}^{n-1}} v_U \cdot \mathbf{n} \rho\, dx'
	= \int_{\mathbf{R}^n_+} (\operatorname{div}v_U) \rho\, dx
	- \int_{\mathbf{R}^n_+} v_U \cdot \nabla\rho\, dx.
\end{equation}
By Proposition \ref{HA} there is an extension operator $Ext: L^1(\mathbf{R}^{n-1}) \to F^1_{1,2}(\mathbf{R}^n)$ such that $Tr \circ Ext$ is the identity operator on $L^1$.
 For $\varphi \in C^\infty_c \left(B_\frac{1}{2}(x'_0)\right)$ we set $\sigma=Ext\,\varphi$.
 By multiplying a cut off function $\theta \in C^\infty_c(U)$ such that $\theta \equiv 1$ in $\frac{1}{2}U$ and consider $\rho=\theta\sigma$ we still find $\rho \in F^1_{1,2}(\mathbf{R}^n)$ by a multiplier theorem \cite[Theorem 3.18]{Sa}, \cite[Section 4.2.2]{Tr92}.
 We estimate \eqref{INT} to get
\begin{align*}
	\left| \int_{\mathbf{R}^{n-1}} v_U \cdot \mathbf{n} \rho\, dx' \right|
	&\leq \left| \int_U (\operatorname{div}v_U) \rho\, dx \right| 
	+ \left| \int_{\mathbf{R}^n_+} v'_U \cdot \nabla' \rho\, dx \right|\\
	&+ \left| \int_{\mathbf{R}^n} v_{U^n} \frac{\partial\rho}{\partial x_n} dx \right| = I + I\hspace{-0.3em}I + I\hspace{-0.3em}I\hspace{-0.3em}I.
\end{align*}
We may assume that $\rho$ is even in $x_n$ by taking $\left(\rho(x', x_n)+\rho(x', -x_n)\right)/2$ so that the second term is estimated by $bmo$-$h^1$ duality $(h^1)^* = (F^0_{1,2})^* = F^0_{\infty,2} = bmo$ as follows
\begin{align*}
	I\hspace{-0.3em}I 
	&= \left| \int_{\mathbf{R}^n_+} v'_U \cdot \nabla' \rho\, dx \right|
	= \frac{1}{2} \left| \int_{\mathbf{R}^n} v'_U \cdot \nabla' \rho\, dx \right| \\
	&\leq C \| v'_U \|_{bmo} \| \nabla' \rho \|_{h^1}.
\end{align*}
The third term is estimated as
\[
I\hspace{-0.3em}I\hspace{-0.3em}I \leq C \| v_{U^n} \|_{bmo} \left\| \frac{\partial\rho}{\partial x_n} \right\|_{h^1}.
\]
The first term is estimated by
\begin{align*}
		I &\leq \| \operatorname{div}v_U \|_{L^n(U)} \| \rho \|_{L^{n/(n-1)}(U)} \\
		&\leq C \| \operatorname{div}v \|_{L^n_\mathrm{ul}(\Gamma_\delta)} 
		\| \nabla \rho \|_{L^1(U)}
\end{align*}
by the Sobolev inequality.
 Since $\| \nabla\rho \|_{L^1} \leq \| \nabla\rho \|_{h^1}$ and $\| \nabla\rho \|_{h^1} \leq \| \rho \|_{F^1_{1,2}} \leq C\|\varphi\|_{L^1\left(B_\frac{1}{2}(x'_0)\right)}$, collecting these estimates yields
\[
	\left| \int_{B_\frac{1}{2}(x_0^{'})} v \cdot \mathbf{n} \varphi\; dx' \right| 
	\leq C \| \varphi \|_{L^1\left(B_\frac{1}{2}(x'_0)\right)} \left( \| v \|_{vbmo^{\mu,\nu}_\delta(\mathbf{R}^n_+)} + \| \operatorname{div}v \|_{L^n_\mathrm{ul}(\Gamma_\delta)} \right).
\]
This yields the desired estimate since $C^\infty_c\left(B_\frac{1}{2}(x'_0)\right)$ is dense in $L^1\left(B_\frac{1}{2}(x'_0)\right)$ and $C$ in the right-hand side is independent of $x'_0 \in \mathbf{R}^{n-1}$.
\end{proof}

We now consider a curved domain.
 Let $\Omega$ be a uniformly $C^2$ domain in $\mathbf{R}^n$ so that the reach $R_*$ of $\Gamma$ is positive and $\beta \in (0,1)$.
\begin{theorem} \label{NTG}
Let $\Omega$ be a uniformly $C^{2+\beta}$ domain in $\mathbf{R}^n$ with $n \geq 2$.
 Let $\mu,\nu,\delta$ be in $(0, \infty]$.
 Then there is a constant $C=C(\mu,\nu,\delta,\Omega)$ such that
\[
	\| v \cdot \mathbf{n} \|_{L^\infty(\Gamma)}
	\leq C \left( \| v \|_{vbmo^{\mu,\nu}_\delta(\Omega)} + \| \operatorname{div}v \|_{L^n_\mathrm{ul}(\Gamma_\delta)} \right)
\]
for all $v \in vbmo^{\mu,\nu}_\delta(\Omega)$.
\end{theorem}
We shall prove this result by localizing the problems near the boundary and use normal (principal) coordinates. Let $\Omega$ be a uniformly $C^{2+\beta}$ domain. In other words, there exist $r_\ast, \delta_\ast > 0$ such that for each $z_0 \in \Gamma$, up to translation and rotation, there exists a function $h_{z_0} \in C^{2+\beta}(B_{r_\ast}(0^{'}))$ with
\begin{eqnarray*}
& & | \nabla^k h_{z_0} | \leq L \; \text{ in } \, B_{r_\ast}(0^{'}) \, \text{ for } \, k=0,1,2, \\
& & [ \nabla^2 h_{z_0} ]_{C^\beta(B_{r_\ast}(0^{'}))} < \infty, \, \nabla^{'} h_{z_0} (0^{'}) = 0^{'}, \, h_{z_0} (0^{'}) = 0
\end{eqnarray*}
such that the neighborhood
\[
U_{r_\ast,\delta_\ast,h_{z_0}}(z_0) := \{ (x^{'},x_n) \in \mathbf{R}^{n} \, | \, h_{z_0}(x^{'}) - \delta_\ast < x_n < h_{z_0}(x^{'}) + \delta_\ast, \, |x^{'}| < r_\ast \}
\]
satisfies
\[
\Omega \cap U_{r_\ast,\delta_\ast,h_{z_0}}(z_0) = \{ (x^{'},x_n) \in \mathbf{R}^{n} \, | \, h_{z_0}(x^{'}) < x_n < h_{z_0}(x^{'}) + \delta_\ast, \, |x^{'}| < r_\ast \}
\]
and
\[
\partial \Omega \cap U_{r_\ast,\delta_\ast,h_{z_0}}(z_0) = \{ (x^{'},x_n) \in \mathbf{R}^{n} \, | \, x_n = h_{z_0}(x^{'}), \, |x^{'}| < r_\ast \}.
\]

For $x \in \Omega$, let $\pi x$ be a point on $\Gamma$ such that $|x - \pi x| = d_{\Omega}(x)$. If $x$ is within the reach of $\Gamma$, then this $\pi x$ is unique. There exist $r < r_\ast$ and $\delta < \delta_\ast$ such that
\[
U(z_0) = \{ x \in U_{r_\ast,\delta_\ast,h_{z_0}}(z_0) \, | \, (\pi x)^{'} \in B_r(0^{'}), \, d_{\Gamma}(x) < \delta \}
\]
is contained in $U_{r_\ast,\delta_\ast,h_{z_0}}(z_0)$. Since $d_\Omega$ is $C^{2+\beta}$ in $\overline{\Gamma}_\sigma$ for $\sigma < R_*$ \cite[Chap.\,14, Appendix]{GT} \cite[\S 4.4]{KP}, we may take $\delta$ smaller (independent of $z_0$) so that $d_\Omega$ is $C^{2+\beta}$ in $\overline{U(z_0)\cap\Omega}$.

We next consider the normal coordinate in $U(z_0)$
\begin{eqnarray}
	\left \{
\begin{array}{l} \label{NC}
	x' = y' + y_n \nabla' d_\Omega \left(y', \psi(y') \right) \\
	x_n = \psi(y') + y_n \partial_{x_n} d_\Omega \left(y', \psi(y') \right)
\end{array}
\right.
\end{eqnarray}
or shortly
\[
	x = \pi x  - d_\Omega (x) \mathbf{n} (\pi x).
\]
Let this coordinate change be denoted by $x=\psi(y)$, $\psi \in C^{1+\beta}(B_r(0^{'}))$. Notice that $\nabla \psi (0) = I$. If we consider $r$ and $\delta$ small, this coordinate change is indeed a local $C^1$-diffeomorphism which maps $U(z_0)$ to $V$ where $V := B_r(0^{'}) \times (-\delta,\delta)$. Moreover, by \cite{Lew}, we extend $\psi$ to a global $C^1$-diffeomorphism $\widetilde{\psi}$ such that $\widetilde{\psi} \, |_V = \psi$ and $\| \nabla \widetilde{\psi} \|_{L^\infty(\mathbf{R}^n)} < 2$. Let the inverse of $\psi$ in $V$ be denoted by $\phi$, i.e. $\phi=\psi^{-1}$.
\begin{lemma} \label{VFG}
	Let $W$ be a vector field with measureble coefficient in $\Gamma_\sigma$, $\sigma<R_*$ of the form
\[
	W = \sum^n_{i=1} w_i \frac{\partial}{\partial x_i}.
\]
Let $y$ be the normal coordinate such that $y_n = d_\Omega(x)$.
 Let $\tilde{W}$ be $W$ in $y$ coordinate of the form $\overline{W}= \sum^n_{j=1} \tilde{w}_j(y) \partial/\partial y_j$. 
 Then
\[
	\tilde{w}_n(y) = \nabla d_\Omega \left(x(y)\right) \cdot w\left(x(y)\right). 
\]
\end{lemma}

We shall prove this lemma in Appendix which follows from a simple linear algebra.
\begin{proof}[Proof of Theorem \ref{NTG}]
We first observe that the restriction $v$ on $U(z_0) \cap \Omega$ is in $bmo^\infty_\infty \left( U(z_0)\cap\Omega \right)$. By considering the following equivalent definition of the seminorm $[f]_{BMO^\infty(D)}$ where
\[
[f]_{BMO^\infty(D)} = \underset{B_r(x) \subset D}{\mathrm{sup}} \, \underset{c \in \mathbf{R}}{\mathrm{inf}} \,\frac{1}{|B_r(x)|} \int_{B_r(x)} | f(y) - c | \, dy,
\]
see \cite[Proposition 3.1.2]{Gra}, we can deduce that the space $bmo^\infty_\infty$ on a bounded domain is independent of bi-Lipschitz coordinate change. We introduce normal coordinate for a vector field $v=\sum^n_{i=1} v_i \partial/\partial x_i$ with $v_i \in bmo^\infty_\infty \left( U(z_0)\cap\Omega \right)$. Let $w$ be the transformed vector field of normal coordinate $y$. By Lemma \ref{VFG}, $w_n$ of $w=\sum^n_{i=1} w_i \partial/\partial y_i$ fulfills $w_n = \nabla d_\Omega \left( x(y) \right) \cdot v \left( x(y) \right)$. Since $v \in vbmo^{\mu,\nu}_\delta(\Omega)$, this implies that $w \in bmo^\infty_\infty(V\cap \mathbf{R}_{+}^n)$ and moreover,
\[
\underset{\ell < \delta, \, B_\ell(x) \subset V}{\mathrm{sup}} \, \ell^{-n} \int_{B_\ell(x) \cap \mathbf{R}_{+}^n} | w_n | \, dy < \infty.
\]
Thus, as in the proof of Theorem \ref{NTH} the zero extension of $w_n$ for $y_n < 0$ is in $bmo^\infty_\infty(V)$, we still denote this extension by $w_n$. Let $J=J(y)$ denote the Jacobian of the mapping $y \mapsto x$ in $V$. For tangential part $w^{'}$ of $w=(w^{'},w_n)$, we take an even extension with weight $J$ of the form
\begin{eqnarray}
	\hat{w}'(y', y_n) = \left \{
\begin{array}{ll}
	w'(y', y_n), & y_n > 0 \\
	w'(y', -y_n)J(y',-y_n)/J(y',y_n), & y_n < 0
\end{array}
\right.
\end{eqnarray}
and set $\tilde{w}(y', y_n)=(\hat{w}', w_n)$. Let $\overline{w}^{'}$ denote the normal even extension of $w^{'}$ to $V$, thus $w \in bmo_\infty^\infty(V \cap \mathbf{R}_{+}^{n})$ implies that $\overline{w}^{'} \in bmo_\infty^\infty(V)$. Let $f$ be the function defined on $V$ such that $f \equiv 1$ for $y_n \geq 0$ and $f = J(y',-y_n)/J(y',y_n)$ for $y_n <0$. Since $J(y)^{-1} = | \mathrm{det} \, D\psi(y) |^{-1} = | \mathrm{det} \, D\phi (\psi(y)) |$ for $y \in V$, we have that $f \in C^\beta(V)$. Notice that $\hat{w}'(y) = \overline{w}' (y) f(y)$, therefore by Theorem \ref{PMU}, we can deduce that $\tilde{w}$ belongs to $bmo_\infty^\infty(V)$. By Theorem \ref{PJM}, the Jones' extension $w_U$ of $\tilde{w}$ belongs to $bmo^\infty_\infty(\mathbf{R}^n)$. Its expression in $x$ coodinate is $v_U$ which is only defined near $\Gamma$. 

If the support of $\rho$ is in $U(z_0)$, then integration by parts implies that
\begin{equation} \label{INT2}
	\int_\Gamma v_U \cdot \mathbf{n} \rho\, d\mathcal{H}^{n-1}
	= \int_\Omega (\operatorname{div}v_U) \rho\, dx
	- \int_\Omega v_U \cdot \nabla\rho\, dx.
\end{equation}
We shall estimate the left-hand side as in the case of $\mathbf{R}^n_+$. The first integral in the right-hand side can be estimated similarly as in the proof of Theorem \ref{NTH}. It is sufficient to only consider the second integral. Let $\Psi: B_r(0') \to \Gamma \cap U(z_0)$ by $(y',0) \mapsto (y',h_{z_0}(y'))$. Extend $h_{z_0} \in C^2(B_r(0'))$ to $\tilde{h} \in C_c^2(\mathbf{R}^{n-1})$ such that $\tilde{h} \, |_{B_r(0')} = h_{z_0}$. Define $\widetilde{\Psi} : \mathbf{R}^{n-1} \to \tilde{h}(\mathbf{R}^{n-1})$ by $(y',0) \mapsto (y',\tilde{h}(y'))$. Hence $\widetilde{\Psi} \, |_{B_r(0')} = \Psi$. Extend further $\widetilde{\Psi}$ to $\widetilde{\Psi}^\ast : \mathbf{R}^n \to \mathbf{R}^n$ by $(y',d) \mapsto \widetilde{\Psi}(y',0) + (0',d)$. Notice that this $\widetilde{\Psi}^\ast$ is a global $C^2$-diffeomorphism whose derivatives are bounded in $\mathbf{R}^n$ up to second-order. Let $\zeta>0$ be a constant, for $\varphi \in C_c^1(\Gamma \cap \zeta U(z_0))$, we have that $\varphi \circ \Psi \in C_c^1(B_\zeta (0'))$. Let $\widetilde{\sigma} = \mathrm{Ext} \, (\varphi \circ \Psi)$ as in the proof of Theorem \ref{NTH} and let $\sigma = \widetilde{\sigma} \circ (\widetilde{\Psi}^\ast)^{-1}$. With this choice of $\sigma$, we have that for $(y',h_{z_0}(y')) \in \Gamma \cap \zeta U(z_0)$,
\[
\sigma(y',h_{z_0}(y')) = \widetilde{\sigma} \circ (\widetilde{\Psi}^\ast)^{-1} (y',h_{z_0}(y')) = \widetilde{\sigma} (y',0) = \varphi \circ \Psi (y',0) = \varphi (y',h_{z_0}(y')).
\]
Thus $\varphi$ is an extension of $\sigma$. Since $(\widetilde{\Psi}^\ast)^{-1}$ is a global $C^2$-diffeomorphism and $\widetilde{\sigma} \in F_{1,2}^1(\mathbf{R}^n)$, we observe that $\sigma \in F_{1,2}^1(\mathbf{R}^n)$, see e.g. see Proposition \ref{CCL} or \cite[Section 4.3.1]{Tr92}.

For each $z_0 \in \Gamma$, there exists $\epsilon_{z_0} > 0$ such that we can find a cutoff function $\theta_{z_0} \in C_c^\infty(U(z_0))$ for which $\theta_{z_0} \equiv 1$ within $\epsilon_{z_0} U(z_0)$ and 
\[
\sum_{|\alpha| \leq 2} \| D^\alpha \theta_{z_0} \|_{L^\infty(\mathbf{R}^n)} \leq M
\]
for some fixed universal constant $M>1$ independent of $z_0$. By multiplying this cutoff function $\theta_{z_0}$, we have that $\rho = \theta_{z_0} \sigma \in F_{1,2}^1(\mathbf{R}^n)$ and $\| \rho \|_{F_{1,2}^1(\mathbf{R}^n)} \leq M \cdot \| \sigma \|_{F_{1,2}^1(\mathbf{R}^n)}$. Hence we take the constant $\zeta$ above to be $\epsilon_{z_0}$.

By coodinate change, we observe that
\[
	\int_\Omega v_U \cdot \nabla\rho\, dx
	= \int_{U(z_0) \cap \Omega} \sum^n_{i=1} v_i \frac{\partial}{\partial x_i} \rho\, dx
	= \int_{V \cap \mathbf{R}_{+}^n} \sum^n_{j=1} w_{U_j}(y)  J(y) \frac{\partial}{\partial y_j} \bigl( \rho \circ \psi (y) \bigr) dy.
\]
The $n$-th component equals
\[
	\int_{V \cap \mathbf{R}_{+}^n} w_{U_n}(y) J(y) \frac{\partial}{\partial y_n} \bigl( \rho \circ \psi (y) \bigr) dy
	= \int_{V} w_{U_n}(y) J(y) \frac{\partial}{\partial y_n} \bigl( \rho \circ \psi (y) \bigr) dy
\]
since $w_{U_n}$ equals zero for $y_n<0$. Consider extensions of H$\ddot{\mathrm{o}}$lder functions \cite{McS} and local diffeomorphism \cite{Lew}, by the $F^0_{1,2}-F^0_{\infty,2}$ duality \cite[Theorem 3.22]{Sa} and Proposition \ref{CC}, we conclude that
\begin{eqnarray*}
\left| \int_V w_{U_n} (y) J(y) \frac{\partial}{\partial y_n} \bigl( \rho \circ \psi (y) \bigr) dy \right| &\leq& C \cdot \sum_{i=1}^n \| w_{U_n} \|_{bmo} \cdot \| J \|_{C^\beta(V)} \cdot \| \partial_{y_n} \psi \|_{C^{\beta}(V)} \cdot \| \nabla \rho \circ \widetilde{\psi} \|_{h^1} \\
&\leq& C \cdot \| w_{U_n} \|_{bmo} \cdot \| \nabla \rho \|_{h^1}.
\end{eqnarray*}

For tangential part we may assume that
\begin{equation} \label{EEV}
	(\rho \circ \psi) (y', y_n) = (\rho \circ \psi) (y', -y_n) 
	\; \, \text{ for } \; y_n <0.
\end{equation}
In fact, for a given $\rho$ we take
\[
g(y',y_n) = \bigl(\rho \circ \psi (y',y_n) + \rho \circ \psi (y',-y_n) \bigr)/2
\]
which satisfies evenness $g(y',y_n) = g(y',-y_n)$ and 
\[
g(y',0) = \theta \circ \psi (y',0) \cdot \sigma \circ \psi (y',0) = \theta(y',h_{z_0}(y')) \cdot \varphi(y',h_{z_0}(y')).
\]
It suffices to take $\rho$ such that $\rho \circ \psi (y) = g(y)$. Thus, we may assume that $\rho \circ \psi$ is even in $y_n$ so that $\partial_{y_j} ( \rho \circ \psi)$ is also even in $y_n$ for $j = 1, 2, \ldots, n-1$. Since $w_{U_j} J$ is even in $y_n$ for $y$ in $V$, we observe that
\[
	\int_{V \cap \mathbf{R}_{+}^n} w_{U_j} (y) J(y) \frac{\partial}{\partial y_j} \bigl( \rho \circ \psi \bigr) \, dy
	= \frac{1}{2} \int_{V} w_{U_j} (y) J(y) \frac{\partial}{\partial y_j} \bigl( \rho \circ \psi \bigr) \, dy
\]
for $1 \leq j \leq n-1$. Similar to the case for the $n$-th component, we thus conclude that
\[
	\left| \int_V w_{U_j} (y) J(y) \frac{\partial}{\partial y_j} (\rho \circ \psi) \, dy \right|
	\leq C \cdot \| w_{U_j} \|_{bmo} \cdot \| \nabla \rho \|_{h^1}.
\]

Collecting these estimates, we conclude that
\begin{align*}
	\left| \int_{\Gamma \cap \epsilon_{z_0} U(z_0)} v \cdot \mathbf{n}\, \varphi\; dx^{n-1} \right|
	&\leq C \| w_U \|_{bmo} \| \nabla\rho \|_{h^1} \\
	&\leq C \| v \|_{vbmo^{\mu,\nu}_\delta(\Omega)} \| \varphi \|_{L^1( \Gamma \cap \epsilon_{z_0} U(z_0))}.
\end{align*}
Thus $\| v \cdot \mathbf{n} \|_{L^\infty} \leq C \|v\|_{vbmo^{\mu,\nu}_\delta(\Omega)}$.
\end{proof}
\begin{remark} \label{NTGB}
\begin{enumerate}
\item[(i)] Since $BMO_b^{\mu,\nu} \subset vbmo_\delta^{\mu,\nu}$ for $\delta < \infty$, the estimate in Theorem \ref{NTG} holds if we replace $vbmo^{\mu,\nu}_\delta$ by $BMO^{\mu,\nu}_b$. Moreover, since we are able to use zero extension in this case. We can follow the proof of Theorem \ref{NTG} directly without the necessity to invoke normal coordinates. We shall state a version of Theorem \ref{NTG} for $BMO_b^{\mu,\nu}$ in the end of this section.
\item[(i\hspace{-0.1em}i)] By Theorem \ref{V1} we may replace $vbmo^{\mu,\nu}_\delta$ by $vBMO^{\mu,\nu}(\Omega)$ in the estimate in Theorem \ref{NTG} since we may always take $\delta\leq\nu < R_\ast$ provided that $\Omega$ is a bounded or an exterior domain.
\end{enumerate}
\end{remark}

\begin{theorem} \label{TRBB}
Let $\Omega$ be a uniformly $C^{1+\beta}$ domain in $\mathbf{R}^n$ with $n \geq 2$. Let $\mu,\nu,\delta$ be in $(0,\infty]$. Then there is a constant $C=C(\mu,\nu,\delta,\Omega)$ such that 
\[
\| v \cdot \mathbf{n} \|_{L^\infty(\Gamma)} \leq C \cdot ( \| v \|_{BMO_b^{\mu,\nu}(\Omega)} + \| \mathrm{div} \, v \|_{L_{\mathrm{ul}}^n(\Gamma_\delta)} )
\]
for all $v \in BMO_b^{\mu,\nu}(\Omega)$.
\end{theorem}

\begin{proof}
For $z_0 \in \Gamma$, let $U(z_0) = U_{r_\ast, \delta_\ast, h_{z_0}}(z_0)$ with $\delta_\ast \leq R_\ast$. We then follow the proof of Theorem \ref{NTG} without invoking the normal coordinates. For $v \in BMO_b^{\mu,\nu}(\Omega)$, let $v_0$ be the zero extension of $v$. We have that $v_0 \in bmo_\infty^\infty(U(z_0))$. Let $v_U$ be the Jones' extension of $r_{U(z_0)} v_0$ by Theorem \ref{PJM} where $r_{U(z_0)} v_0$ denotes the restriction of $v_0$ on $U(z_0)$. For $\varphi \in C_c^1(\Gamma \cap \frac{1}{2} U(z_0))$, we construct the function $\sigma$ in the same way as in the proof of Theorem \ref{NTG}. Since the boundary $\Gamma$ is uniformly $C^{1+\beta}$, $\widetilde{\Psi}^\ast$ is a global $C^{1+\beta}$-diffeomorphism. By Proposition \ref{CCL}, we have that $\sigma = \widetilde{\sigma} \circ (\widetilde{\Psi}^\ast)^{-1} \in F_{1,2}^1(\mathbf{R}^n)$. Pick $\theta$ in $C_c^\infty(U(z_0))$ such that $\theta \equiv 1$ within $\frac{1}{2} U(z_0)$ and let $\rho = \theta \sigma$, we deduce that $\rho \in F_{1,2}^1(\mathbf{R}^n)$ and 
\[
\left| \int_\Omega v_U \cdot \nabla \rho \, dx \right| \leq C \cdot \| v_U \|_{bmo} \cdot \| \nabla \rho \|_{h^1} \leq C \cdot \| v \|_{BMO_b^{\mu,\nu}(\Omega)} \cdot \| \nabla \rho \|_{h^1}.
\] 

Therefore,
\[
\left| \int_{\Gamma \cap \frac{1}{2} U(z_0)} v \cdot \mathbf{n}\, \varphi\; dx^{n-1} \right| \leq C \cdot \| v \|_{BMO_b^{\mu,\nu}(\Omega)} \cdot \| \varphi \|_{L^1( \Gamma \cap \frac{1}{2} U(z_0))}.
\]
The proof is therefore complete.
\end{proof}
%
\section{Appendix} \label{APP} 

We shall prove Lemma \ref{VFG}.
 We first recall a simple property of a matrix.
\begin{proposition} \label{LA}
Let $A$ be an invertible matrix
\[
	A = (\vec{a}_1, \ldots, \vec{a}_n) 
\]
when $\vec{a}_j =^t(a_{ij})_{1\leq i \leq n}$ is an column vector. 
 Assume that $\vec{a}_n$ is a unit vector and orthogonal to $\vec{a}_j$ with $1 \leq j \leq n-1$.
 Then $n$-row vector of $A^{-1}$ equals $^t\vec{a}_n$.
 In other words, if one writes $A^{-1} = (b_{ij})_{1\leq i,j \leq n}$, then $b_{nj}=a_{jn}$ for $1\leq j \leq n$.
\end{proposition}
\begin{proof}
By definition the row vector $\vec{b}=(b_{nj})_{1 \leq j \leq n}$ must satisfies $\vec{b}\cdot\vec{a}_j=0$ ($j=1,\ldots,n-1$), $\vec{b}\cdot\vec{a}_n=1$.
 Since $\left\{\vec{a}_j\right\}^{n-1}_{j=1}$ spans $\mathbf{R}^{n-1}$ orthogonal to $\vec{a}_n$, first identities imply that $\vec{b}$ is parallel to $\vec{a}_n$.
We thus conclude that $\vec{b}=\vec{a}_n$ since $\vec{b}\cdot\vec{a}_n=1$ and $|\vec{a}_n|=1$.
\end{proof}
\begin{proof}[Proof of Lemma \ref{VFG}]
We recall the explicit representation \eqref{NC} of the normal coordinate.
 The Jacobi matrix from $y \longmapsto x$ is of the form
\[
	A = (\vec{a}_1, \ldots, \vec{a}_n)
\]
with $\vec{a}_j =^t \bigl(\delta_{ij} - y_n \partial_j \mathbf{n}_i \left(y',\psi(y')\right),\ \partial_j \psi(y') - y_n \partial_j \mathbf{n}_n \left(y',\psi(y')\right) \bigr)_{1 \leq i \leq n-1}$, $1 \leq j \leq n-1$,
\[
	\vec{a}_n = -^t\mathbf{n} \left(y',\psi(y')\right) \quad\text{where}\quad
	\mathbf{n} = -\nabla d_\Omega.
\]
Note that the vector $\left(\delta_{ij}, \partial_j \psi(y') \right)_{1\leq i \leq n-1}$ is a tangential vector to $\Gamma$.
 Moreover, $(\partial_j \mathbf{n}_1, \ldots, \partial_j \mathbf{n}_n)$ is also tangential since $\partial_j \mathbf{n}\cdot \mathbf{n}=\partial_j |\mathbf{n}|^2/2=0$.
 Thus $\vec{a}_j$ is orthogonal to $\vec{a}_n$ for $1\leq j \leq n-1$.
The invertibility of $A$ is guaranteed if $y_n<R_*$.

By a chain rule we have
\begin{align*}
	\overline{w} &= \sum^n_{j=1} \tilde{w_j} (y) (\partial/\partial y_j) \\
	&= \sum^n_{i=1} \sum^n_{j=1} \tilde{w_j} \frac{\partial x_i}{\partial y_j} \frac{\partial}{\partial x_i} 
\end{align*}
so that
\[
	w_i \left(x(y)\right) = \sum^n_{j=1} \tilde{w_j} (y) \frac{\partial x_i}{\partial y_j} 
	\quad\text{i.e.,}\quad w = A \tilde{w},
\]
where $A=(\partial x_i/\partial y_j)_{1\leq i,j \leq n}$, $\tilde{w}=^t(\tilde{w_i}, \ldots, \tilde{w_n})$, 
$w=^t(w_1, \ldots, w_n)$.
 Thus
\[
	\tilde{w} = A^{-1} w.
\]
By Proposition \ref{LA}, the last row of $A^{-1}$ equals $\nabla d_\Omega$.

We thus conclude that $\tilde{w}_n = \nabla d_\Omega \cdot w$. This is what we would like to prove.
\end{proof}
\section*{Acknowledgement}
The first author is partly supported by the Japan Society for the Promotion of Science through grants Kiban A (No. 19H00639), Challenging Pioneering Research (Kaitaku) (No. 18H05323), Kiban A (No. 17H01091). The second author is partly supported by the Mizuho international foundation through foreign students scholarship program.

\end{document}